\documentclass[reqno,a4paper,12pt]{amsart}

\usepackage{amsmath}
\usepackage{amssymb}
\usepackage{mathrsfs}
\usepackage{amsthm}
\usepackage{esint}
\usepackage{amsopn}

\usepackage{graphicx}
\usepackage{caption}
\usepackage{subcaption}
\usepackage[margin=2.5cm,top=3.5cm,bottom=4cm,footskip=1cm,headsep=1cm]{geometry} 
\DeclareMathOperator*{\argmin}{\arg \min}%
\newcommand{\tv}[1]{|#1|_{BV(\Omega)}}%
\newcommand{\dx}{\:\mathrm{d}x}

 \newtheorem{theorem}{Theorem}
  \newtheorem{lemma}[theorem]{Lemma}
  
  \newtheorem{proposition}[theorem]{Proposition}
  \newtheorem{definition}[theorem]{Definition}
\newtheorem{rem}[theorem]{Remark}

\title[Motion-Corrected Reconstruction of Density Images]{A Nonlinear Variational Approach to \\ Motion-Corrected Reconstruction of Density Images} 

\author{Martin Burger$^{\dag \S}$ \and Jan Modersitzki$^{\ddag \P}$ \and  Sebastian Suhr$^{\dag \ddag}$}
\thanks{$^\dag$ Institute for Computational and Applied Mathematics, University of M\"unster, Einsteinstr. 62, 48149
M\"unster, Germany.}
\thanks{$^\S$ Cells in Motion Cluster of Excellence, University of M\"unster}
\thanks{$^\ddag$ Institute of Mathematics and Image Computing, University of L\"ubeck, Maria-Goeppert-Straße 3, 23562 L\"ubeck, Germany}
\thanks{$^\P$ Fraunhofer MEVIS Project Group Image Registration, Maria-Goeppert-Straße 3, 23562 L\"ubeck, Germany\thanks{This work was
supported by the Deutsche Forschungsgemeinschaft (DFG) through grants BU 2327/8-1 and  MO-1053/2-1, as well as by ERC via Grant EU FP 7 - ERC Consolidator Grant
615216 LifeInverse. MB acknowledges further support by the German Science Foundation DFG via
EXC 1003 Cells in Motion Cluster of Excellence, M\"unster, Germany}}
\begin{document}
\maketitle

\begin{abstract}
The aim of this paper is to establish a nonlinear variational approach to the reconstruction of moving density images from indirect dynamic measurements. Our approach is to model the dynamics as a hyperelastic deformation of an initial density including preservation of mass. Consequently we derive a variational regularization model for the reconstruction, which - besides the usual data fidelity and total variation regularization of the images - also includes a motion constraint and a hyperelastic regularization energy.

Under suitable assumptions we prove the existence of a minimizer, which relies on the concept of weak diffeomorphisms for the motion. Moreover, we study natural parameter asymptotics and regularizing properties of the variational model. Finally, we develop a computational solution method based on alternating minimization and splitting techniques, with a particular focus on dynamic PET. The potential improvements of our approach compared to conventional reconstruction techniques are investigated in appropriately designed examples.
\end{abstract}

\pagestyle{myheadings}
\thispagestyle{plain}

\section{Introduction}
With current advances of imaging devices, in particular in biomedical imaging, dynamic studies including motion attract more and more attention. Prominent examples are live microscopy, MRI, and emission tomography (PET and SPECT). For the latter  cardiac applications are particularly interesting, which are subject to cardiac and respiratory motion. In most of such applications a density image (e.g. the density of a radiochemical or fluorescent tracer), which is subject to motion between consecutive time steps, is to be reconstructed from indirect measurements. This problem shall be tackled with a novel variational approach in this paper. 

The need for advanced motion-corrected reconstruction methods is caused by the fact that separate reconstructions in small time intervals - in which the motion effect is negliglible - are of inferior quality. The latter problem is either caused by severe undersampling (e.g. in MRI due to time restrictions in acquiring slice data) or the bad signal-to-noise ratio (e.g. in optical imaging and emission tomography techniques being based on counting photons) in small time steps. Several recent approaches exploiting sparsity properties of the images to be reconstructed achieve strong improvements in these two situations (cf. MR compressed sensing \cite{lustig2007sparse}, Bregman-EM-TV for PET \cite{muller2011reconstruction}). However, those approaches are still limited in certain practical situations, and since separate reconstruction in small time scales does not exploit natural temporal correlation, spatio-temporal image reconstruction methods are an obvious next step.

Several approaches tried to obtain improved reconstructions by appropriately averaging reconstructions at different time steps. To do so, the motion between those time steps needs to be estimated, which can be performed either by flow-type (optical flow) techniques  \cite{Dawood2006} or by image registration  \cite{Klein2000}. Subsequently the motion in the image can be corrected and improved images can be obtained by averaging, cf. e.g. \cite{Klein2000} for an implementation in a clinical PET setup. Further improvements are often achieved by iterating a reconstruction step (in separate time steps) and the motion estimation, where some motion-corrected density (inital or average) serves as a prior for the next reconstruction step. A remaining disadvantage are artefacts that can be created this way as well as unclear convergence properties of such alternating iteration approaches. Those can be cured by formulating an appropriate variational model, which is then minimized by an alternating iteration method. In this way consistent and convergent iteration schemes can be obtained and the properties of the limits can be understood from the structure of the underlying variational problem. A key issue is that smoothness and smallness of deformations or velocities can be used as reasonable priors (via appropriate energy terms in the variational problems), which finally leads to physically reasonable reconstructions with significantly improved signal-to-noise ratio. Examples 
are the approaches of \cite{B.A.Mair2006,Blume2010TMI,Jacobson2003}  to motion-corrected PET reconstruction.

Similar as in \cite{B.A.Mair2006,Blume2010TMI} we are going to minimize a functional of the form
\begin{equation}
 J(\rho,y)=	\sum_{i=0}^N \left( D(K\rho^i,f^i) + \alpha^i R_I(\rho^i) + \beta^i R_M(y^i) \right)
\end{equation}
for an appropriate transformation model. Here $D$ is the data fidelity between estimated and measured data, and $R_I$ respectively $R_M$ are regularization functionals on the image respectively motion (with nonnegative regularization parameters $\alpha^i$ and $\beta^i$).

In our setup we reconstruct a sequence of nonnegative densities $\rho^0,\ldots,\rho^N$ on $\Omega \subset \mathbb{R}^d$, typically $d=2,3$ such that
\begin{equation}
	\rho^i(x)=\rho^0(y^i(x)) \operatorname{det}(\nabla y^i(x)), 
	\label{eq:transformation_model}
\end{equation}
with a reasonably smooth deformation field $y^i: \mathbb{R}^d \rightarrow \mathbb{R}^d$. The different densities are thought of as evaluations at (ascending) time steps $t^i$.

The measurements $f^i$ are noisy versions of the projected image at the$i-$th time step $K(\rho^i)$, with a stationary forward operator $K$, e.g. a convolution in fluorescence microscopy or versions of the x-ray transform in PET and SPECT. 
Motivated by the above applications, in particular cardiac PET, we will focus on noisy data drawn from Poisson statistics, i.e. $f^i(x)$ is interpreted as a Poisson distributed random variable with  expected value $ \tau ~(K \rho^i)(x)$, where $\tau$ is a typical length of the time interval. However, extension to various other commonly used noise models, such as additive noise or additional background noise, is straight-forward in our variational setting by simply exchanging a data fidelity term related to the negative log-likelihood of the noise distribution. 

The main contributions of the paper are the following:

\medskip

\begin{itemize}
\item An appropriate modelling of the mass-conserving density transformation and its implementation in the reconstruction process, which has not been considered previously. 

\item A model with suitable regularization functionals for images with edges (total variation) and for physically reasonable, (weakly) one-to-one deformations (hyperelastic). 

\item A detailed mathematical analysis of the arising variational problem in a function space setting.

\item The construction and implementation of computational methods for a setup in cardiac PET, with applications to realistic datasets. 
\end{itemize}

\medskip
As we shall see below, a major issue in the analysis of the model is the fact that we only deal with BV images and Sobolev deformations, whose composition is not defined in a straight-forward way and continuity with respect to weak convergence is a rather difficult issue.
The analysis section is therefore split into two main parts: In the first we present some results from geometric measure theory, which allow for a generalized version of the change of variables formula. With help of this formula we are able to prove weak $L^1$ convergence of the composition of $BV$ functions with Sobolev mappings (Theorem \ref{theorem:konvergenz}). This convergence result is the backbone of the existence result we formulate for injective transformations (Theorem \ref{theorem:Existence_mcr_functional}).

The next section lays the focus on our numerical optimization. Since we assume time discrete data, we pick a standard space discretization and perform a First-Discretize-Then-Optimize approach. Similar to \cite{B.A.Mair2006} we end up with an alternating two step minimization: The first step is reconstruction with a fixed motion estimation, while we optimize the motion estimation with a fix reconstructed image in the second. We describe both steps in detail and show that we can perform standard EM-TV algorithms also with motion corrected reconstruction. The minimization in the motion step leads to image registration, but we present a new distance measure which is defined on the detector domain. Despite from that we follow the widely known FAIR-framework \cite{Modersitzki2009} for the registration.

\section{Motion Models and Variational Formulation}

In the following we discuss the appropriate mathematical modelling of motion-corrected reconstruction of density images. Our goal is to obtain the reconstructed image, which is a sequence of nonnegative,integrable probability densities
\begin{equation}
	\rho = (\rho^0,\rho^1,\ldots,\rho^N) \in L_+^1(\Omega)^{N+1}
\end{equation}
as well as the motion (deformation)
\begin{equation}
	y = (y^1,\ldots,y^N) \in H^1(\Omega)^N
\end{equation} 
as a minimizer of a suitable variational problem. We will further restrict the set of admissible deformations below and sometimes also look for densities in different spaces, e.g. as functions of bounded variation or relaxed as probability measures instead of densities. We will investigate the following variational model: Minimize
\begin{equation}
 J(\rho,y) = \sum_{k=0}^N \left( D(K\rho^i,f^i) + \alpha^i \tv{\rho^i} \right) + \sum_{k=1}^N \beta^i S^{hyper}(y^i) 
\label{eq:functional}
\end{equation} 
over nonnegative $\rho$ subject to \eqref{eq:transformation_model}. Here $K:L^1(\Omega) \rightarrow Y$ is assumed to be a compact operator and the $f^i$ are noisy measurements at time $t^i$, we refer to the subsequent sections for precise definition of the BV-seminorm and $S^{hyper}$. This can be interpreted as a MAP estimate for a (formal) Bayesian model as we demonstrate in the following.

\subsection{Bayesian Modelling}

 In order to derive a variational model we can resort to a formal Bayesian model, where for simplicity we use the above notations. For a rigorous modelling, discretized versions of the above variables or infinite-dimensional probability distributions should be used, which is beyond the scope of this paper. Bayes theorem yields the posterior probability density
\begin{equation}
	{\bf P}((\rho,y)|f) \sim {\bf P}(f|\rho) ~{\bf P}(\rho|y)~{\bf P}(y) , 
\end{equation}
where ${\bf P}(f|\rho)$ is the likelihood, which we simply assume as 
\begin{equation}
{\bf P}(f|\rho) = \prod_{i=0}^N {\bf L}_n (K\rho^i,f^i), 
\end{equation}
where ${\bf L}_n$ is the stationary noise likelihood of observing $f$ given $\rho$. Note that in our model we make the natural assumption that ${\bf L}_n$ does not depend on $y$ explicitely, since the image formation and hence the involved noise is considered as an instantaneous snapshot given some density at time $t^i$. For data obtained by collecting information over a larger time interval (or even the full interval $(t^{i-1},t^i]$, the modelling needs to be changed at this point.

The prior probability density for the image sequence given the deformations is given by
\begin{equation}
	{\bf P}(\rho|y) \sim  {\bf P}_0(\rho^0) \prod_{k=1}^N \varepsilon(\rho^i - \rho^0(y^i) \text{det}(\nabla y^i) ) {\bf P}_i(\rho^i),
\end{equation}
where $\varepsilon$ is the concentrated measure (centered at the origin) and ${\bf P}_i$ is some a-priori probability for the image at time step $i$.  Finally, ${\bf P}(y)$ is a prior probability density on the deformation sequence. 

In order to compute a maximum a-posteriori probability (MAP) estimator, we need to restrict to the set of images and deformations satisfying \eqref{eq:transformation_model} due to the concentrated measure and minimize the negative log-likelihood of the remaining factors in the density. With $D(K\rho,f) = - \log 
{\bf L}_n (K\rho^i,f^i)$ we obtain a problem of the form
\begin{equation}
\min\limits_{\rho,y} \sum_{i=0}^N \left( D(K\rho^i,f^i) - \log {\bf P}_i(\rho^i) \right)  -  \log {\bf P}(y).
\end{equation}
As usual the negative logarithms of the prior probabilities are related to regularization terms in standard inverse problems theory \cite{Kaipio2005}. Hence, we can directly specify those using appropriate models. Since we aim at reconstructing images with sharp edges we employ total variation regularization on the image, i.e., 
\begin{equation}
	- \log {\bf P}_i(\rho^i) = \alpha^i \tv{\rho^i}
\end{equation}
For the deformation we use a hyperelastic energy (cf. \cite{Ruthotto2012a}), i.e., 
\begin{equation}
	 -  \log {\bf P}(y) = \sum_{i=1}^N \beta^i S^{hyper}(y^i),
\end{equation}
which is natural since it also regularizes the Jacobian determinant det$(\nabla y^i)$. Note that hyperelastic regularization also enforces orientation preservation and thus $det(\nabla y) >0$ a.e. \cite{Burger2013}.
In detail the hyperelastic energy is given by 

\begin{equation}
S^{\text{hyper}}(y) \ = \ \int\limits_\Omega{\alpha_1 \operatorname{len}(\nabla y) \: + \: \alpha_2 \operatorname{surf}(\operatorname{cof}(\nabla y)) \: + \: \alpha_3 \operatorname{vol}(\text{det}(\nabla y))} \dx \ ,
\end{equation}
with the penalty functions
\begin{equation*}
\operatorname{len}(s) = \|s - \mathcal{I}\|_{\text{Fro}}^2 \ , \qquad \operatorname{surf}(s) = \left(\max(\|s\|_{\text{Fro}}^{2}-3,0)\right)^{2}, \qquad  \operatorname{vol}(s) = \left\{  \begin{array}{ll} \frac{(s-1)^{4}}{s^2} & \text{if } s > 0 \\ + \infty & \text{else} \end{array}  \right. \ .
\end{equation*}

The three terms punish deviations from the identity and in volume, length and surface. This energy enforces locally one-to-one transformations, but there might be globally non injective transformations, see \cite{Suhr2015}.

Having defined regularization for the density as well as for the motion field, we look for a minimizer of \eqref{eq:functional}, subject to \eqref{eq:transformation_model}

In the particular case of Poisson data, the data fidelity equals the Kullback-Leibler divergence (for $K$ being a continuous operator from $L_+^1(\Omega)$ to $L^1_+(\Sigma)$) up to a constant independent of $f^i$, more precisely
\begin{equation}
D(K\rho^i,f^i) = {\int\limits_{\Sigma}{K \rho^i -f^i\log(K \rho^i )d\sigma}} \label{eq:datafidelity}
\end{equation}

\section{Analysis}

This section is devoted to the analysis of the functional we derived in the previous section. Resulting from the mass-preservation condition we derived a functional with transformed images of the form $\rho(y)\operatorname{det}(\nabla y)$. The transformation theorem for integrals is a powerful tool for dealing with such transformations. Unfortunately the hyperelastic regularization we imposed on the transformation does not guarantee diffeomorphic transformations in the classical sense. Thus we describe the relaxation of classical infinitesimal calculus to Sobolev mappings, before we focus on the presentation of our analytical results.

\subsection{Preliminary Results}

In this section we generalize known definitions from the classical infinitesimal calculus to equivalence classes of functions in Lebesque- respectively Sobolev spaces (compare for example \cite{Giaquinta1992,Hajlasz1993,Giaquinta1998a} for a further course on this matter). Our final goal is to derive a version of the transformation theorem for integrals for non-diffeomorphic functions. Since we can not distinguish functions differing on zero sets, the classical definition of differentiability is not a feasible way, because we would like to obtain the same result for all representatives of the equivalence class. A natural way to define a coherent function value for all representatives is via averaging, which leads to the following definition. 

\begin{definition}
Let $\Omega$ be a domain and $ y \in L^1(\Omega)$. Then the set of points $x$ for which $y_l(x)$ exists, such that
\begin{equation}
\fint\limits_{B(x,r)}|y(z)-y_l(x)|\mathrm{d}z \rightarrow 0 \qquad \text{as} \qquad r\rightarrow 0
\end{equation}
is called the Lebesgue set $\mathscr{L}_y$, while the points in $\mathscr{L}_y$ are called Lebesgue points.
\end{definition} 
\begin{rem}\hspace{1 mm} \\
\vspace{-0.5cm}
\begin{itemize}
\item It is known  that the complement of the Lebesgue points is a zero set \cite[Thm. 2.19]{Giaquinta2009}. 
\item For $y \in W^{1,1}_{loc}$ we can define the set of Lebesgue points $\mathscr{L}_{Dy}$ for the derivative $Dy$ analogously. 
\item If $y$ is a vector-valued function, we say $x$ is a Lebesgue point, iff it is a Lebesgue point for each component function.
\end{itemize}
\end{rem}

Differentiability can be generalized in a similar way \cite{Ambrosio2000}:

\begin{definition}[Approximate differential for $L^1_{loc}$]
Let $y\in L^1_{loc}(\Omega,\mathbb{R}^m)$ and let $x\in \mathscr{L}_y$; we say that $y$ is approximately differentiable at $x$, iff there exists a $d\times m$ matrix $L$, such that
\begin{equation}
\fint\limits_{B_r(x)}\frac{|y(z)-y_l(x)-L(z-x)|}{r}\mathrm{d}z \rightarrow 0 \qquad \text{as} \qquad r\rightarrow 0.
\end{equation}
\end{definition}

To overcome difficulties arising from changing functions on a zero set, Nikolai Lusin imposed in his dissertation \cite{Luzin1915} the so called Lusin condition, also known as N-condition: 

\begin{definition}[Lusin's condition]
A mapping $y: \Omega \rightarrow \mathbb{R}^{d}$ satisfies Lusin's condition, iff:
\begin{equation}
\label{eq:lusin_condition}
\forall E \subset \Omega,~ \lambda(E) = 0 \Rightarrow \lambda(y(E))=0
\end{equation}
with $\lambda$ denoting the Lesbesgue-measure. 
\end{definition}

Having surmounted these difficulties we need to take in account, that a non-diffeomorphic function may hit some points several times. For this problem Stefan Banach introduced the so called Banach indicatrix, which gives the number of roots to an equation \cite{Banach1925}. This concept was generalized to discontinuous functions by Lozinski \cite{Lozinski1958} and into higher dimensions by Kronrod \cite{Kronrod1950} and  Vitushkin in his master thesis \cite{Vitushkin1955}:

\begin{definition}[Banach indicatrix]
Let $y:\mathbb{R}^{d} \rightarrow \mathbb{R}^{m}$, $E\subset \mathbb{R}^{d}$. The Banach indicatrix  is a function
\begin{equation*}
N_y(\cdot,\Omega) : \mathbb{R}^{m} \rightarrow \mathbb{N}_{0} \cup \{\infty\},
\end{equation*} 
which is given by
\begin{equation}
N_y(z,E):=\operatorname{card}(\{y^{-1}(z) \cap \Omega\}).
\end{equation}
\end{definition}

We are now ready to present a change of variables formula under minimal assumptions, which was given by Hajlasz \cite{Hajlasz1993}. The central idea is, that points hit multiple times by the transformation need to be taken into account as multiplicative factor.

\begin{theorem}[Area formula]
\label{theorem:area_formula}
Let $y: \mathbb{R}^d \rightarrow \mathbb{R}^{m}$ be a mapping. If $y$ is approximately differentiable almost everywhere, then $y$ can be redefined on a zero set, such that the new $y$ fulfills  Lusin's condition. Furthermore the following statements hold for every measurable subset $\Omega$ and positive measurable function $u:\tilde{\Omega} \rightarrow \mathbb{R}$:

\begin{itemize}
\item[(i)] The functions $u(y)\operatorname{det}(\nabla y)$ and $u(z)N_y(z,\Omega)$ are measurable.
\item[(ii)] If moreover $u \geq 0$ then
\end{itemize}
\begin{equation}
\int\limits_{\Omega}{u(y(x))\operatorname{det}(\nabla y(x))\mathrm{d}x}=\int\limits_{\mathbb{R}^{m}}{u(z)N_y(z,\Omega)\mathrm{d}z}.
\label{eq:area_formula}
\end{equation}
\begin{itemize}
\item[(iii)] If one of the functions $u(y)\operatorname{det}(\nabla y)$ and $u(z)N_y(z,E)$ is integrable then so is the other and the formula \eqref{eq:area_formula} holds.
\end{itemize}

Additionally we have
\begin{equation}
\int\limits_{\Omega}u(x)\operatorname{det}(\nabla y(x))\mathrm{d}x=\int\limits_{\mathbb{R}^{d}}\sum\limits_{w\in ( y^{-1}(z)\cap \Omega)}u(w)\mathrm{d}z.
\end{equation}

\end{theorem}

\begin{proof}
See \cite[Theorem 2]{Hajlasz1993} for the first part of the theorem. The second part is given in \cite[Chapter 1.2 Theorem 2]{Giaquinta1998a} for Lipschitz mappings and can be generalized by following the proof by Hajlasz \cite[Theorem 2]{Hajlasz1993}.  
\end{proof}

As we see, the generalization of the transformation theorem can take into account, that points might be hit several times. Consequently a mapping with no additional injectivity restriction might violate the mass-preservation condition.

Furthermore, this subject ist strongly linked to the topological degree, which was introduced by Browers in 1911 \cite{Brouwer1911} and generalized to Sobolev mappings  by Giaquinta et al.  \cite{Giaquinta1994a}. We will not focus on this topic, but we will use some of this results in the remaining part of this thesis. We start by giving the generalized definition of the topological degree, see also\cite{Giaquinta1994a}

\begin{definition}[Topological degree]
Let $\Omega\subset \mathbb{R}^{d}$ be an open set and $y$ an almost everywhere approximately differentiable map with Jacobian $Dy$ the degree of $y$ is defined as
\begin{equation}
\operatorname{deg}(y,\Omega,z):=\sum\limits_{x\in y^{-1}(z)}\operatorname{sgn}(\operatorname{det}(Dy(x)))
\end{equation}
\end{definition} 

\begin{rem}
The topological degree is strongly related to the Banach indicatrix, since it coincides with the Banach indicatrix for certain mappings. As a direct consequence of  \cite[Chapter 1, Proposition 2]{Giaquinta1994} we obtain for orientation preserving mappings:

\begin{equation}
\operatorname{deg}(y,\Omega,\cdot)=N_y(\cdot,\Omega) \qquad \text{a.e.}
\end{equation}

\end{rem}

An interesting property of the topological degree is that it is completely determined  on the boundary for sufficiently regular functions. This is phrased in the following proposition.

\begin{proposition}
\label{proposition:degree_boundary}
Let $\Omega$ be a bounded Lipschitz domain in $\mathbb{R}^d$ and let $y_1,y_2$ be mappings in $W^{1,d-1}(\Omega, \mathbb{R}^{d})$  with $\operatorname{cof}(\nabla y_i)\in L^{\frac{d}{d-1}}$. Suppose that $y_1=y_2$ text $\partial\Omega$ in the sense of $W^{1,d-1}$ traces. Then
\begin{equation*}
\operatorname{deg}(y_1,\Omega,z)=\operatorname{deg}(y_2,\Omega,z)
\end{equation*}
\end{proposition}
\begin{proof}
See \cite{Giaquinta1994}, Chapter 2, Proposition 1.
\end{proof}

\noindent
To conclude this brief summary, we mention that there is indeed a specification to injective Sobolev functions. This leads to the field of weak diffeomorphisms \cite{Giaquinta1989}. With Theorem \ref{thm:injectivity_closed_constraint} we will only present a compactness result, which is central for our analysis. A short introduction is given in the appendix.
A weak diffeomorphism $y$ can be expressed as the limit of a sequence $y_n$ of orientation preserving $C^1$ diffeomorphisms \cite{Giaquinta1989}. Furthermore closeness and compactness results for this class can be stated (\cite[Chapter 5, Theorems 3 and 4]{Giaquinta1994}). This relies heavily on weak convergence in the space of weak diffeomorphisms $\tilde{\operatorname{dif}}^{p,q}$. Thus we present a similar result published by Henao and Mora-Corral \cite{Henao2010} instead, which follows directly from \cite[Proposition 2 and Theorem 2]{Henao2010}.

\begin{theorem}[Injectivity as closed constraint]
\label{thm:injectivity_closed_constraint}
For each $j\in\mathbb{N}$ let $y_j,u :\Omega\subset \mathbb{R}^{d} \rightarrow \mathbb{R}^m$ be a.e. approximately differentiable. Assume furthermore, that
\begin{equation}
y_j \in W^{1,p}(\Omega,\mathbb{R}^m) \quad p\geq d-1 \qquad \operatorname{det}(D y_j)\in L^1(\Omega)
\end{equation}
as well as
\begin{equation}
\label{eq:cofactor_surface_energy}
\operatorname{cof}(D y_j)\in L^{g}(\Omega) \quad q \geq \frac{p}{p-1} \qquad \sup\limits_{j\in \mathbb{N}} \|\operatorname{cof}(D y_j)\|_1 < \infty.
\end{equation}
 Suppose that there exists $\theta\in L^1(\Omega)$ such that $\theta >0$ a.e. and
 \begin{equation}
 y_j \rightarrow y \qquad \operatorname{det}(D y_j)\rightharpoonup \theta \text{ in } L^1(\Omega) 
 \end{equation}
 as $j\rightarrow \infty$. Assume that for each $j\in \mathbb{N}$ the function $y_j$ we have $N_{y_j}\leq 1$ a.e. with $\operatorname{det}(D y_j)>0$ a.e.. Then
 \begin{itemize}
 \item[(i)] $\theta=|\operatorname{det}(D y)|$ a.e.,
 \item[(ii)] $N_y \leq 1$ a.e..
 \end{itemize}
\end{theorem}

\subsection{Regularization Functionals}

In this section we shortly present the properties of the regularization functionals we use in our reconstruction framework. Both regularizations guarantee stronger convergence properties additionally to the compactness of sublevel sets. We do not focus on general possibilities for choosing these energies, but rather present the specific choices for images (total variation) and transformation (hyperelastic).

Total variation regularization was first introduced for image denoising in \cite{Rudin1992},  now known as the ROF-model. Recently TV regularization has been applied to different tasks in imaging. 
In fact there are many different equivalent ways to define total variation and $BV-$functions (compare \newline \cite[Chapter 10, Definition 10.1.1]{Attouch2006}), but we focus on the common definition in image processing (cf. \cite{Burger2013a}):

\begin{definition}[BV seminorm and functions of bounded variation]
Let 
$u: \Omega \subset \mathbb{R}^{d} \rightarrow \mathbb{R}.$
Then the BV-seminorm is given by
\begin{equation}
\tv{u} := \sup\limits_{g \in C^\infty_0(\Omega;\mathbb{R}^{d}),\|g\|_\infty \leq 1}\int\limits_{\Omega}u \nabla \cdot g\mathrm{d}x.
\label{eq:tv_regularization}
\end{equation}
 Consequently we define the space of functions with bounded variation $BV(\Omega)$ by
\begin{equation}
BV(\Omega):=\{u\in L^1(\Omega)\,|\, \tv{u} < \infty\}  \qquad \|u\|_{BV}:=\|u\|_1+\tv{u}
\end{equation}

\end{definition}

The BV-seminorm is lower semicontinuous in the $L^1_{loc}$ topology \cite[Remark 3.5]{Ambrosio2000}.
 \cite[Remark 3.12]{Ambrosio2000},
thus one can define a weak-star convergence. 

Note that due to compactness of the embedding operator the weak star convergence in $BV$
guarantees strong convergence in $L^1$. As we will see later, this will be useful to prove convergence properties of compositions of functions. Furthermore we can approximate any $BV$-function by a sequence of smooth functions:

\begin{theorem}[Approximation by smooth functions]
\label{theorem:Bv_smooth_approximation}
Let $u\in L^1(\Omega;\mathbb{R}^m)$. Then $u\in BV(\Omega,\mathbb{R}^m)$, if and only if there exists a sequence $u_n$ in $C^\infty(\Omega,\mathbb{R}^m)$ converging to u in $L^1(\Omega,\mathbb{R}^m)$ and there exists a constant $L$ satisfying 
\begin{equation}
 \lim\limits_{n\rightarrow \infty}\int\limits_{\Omega}|\nabla u_n|\mathrm{d}x < \infty
\end{equation}
\end{theorem}
\begin{proof}
\cite[Thm. 3.9]{Ambrosio2000}.
\end{proof}

\subsubsection{Hyperelastic Regularization}

For the  hyperelastic regularization energy, we presented earlier, Ruthotto \cite{Ruthotto2012a} used the following set of admissible transformations 
\begin{equation}
\label{eq:hyperelasmenge}
\mathscr{A} := \{y \in \mathscr{A}_0 : |\int\limits_{\Omega}{y(x)dx}|\leq \operatorname{vol}(\Omega)(M+\operatorname{diam}(\Omega))\},
\end{equation}
where $\Omega$ is bounded by M, and $\mathscr{A}_0$ is defined by:
\begin{align*}
\mathscr{A}_0 := \{ y\in W^{1,2} : &\operatorname{cof}(\nabla y) \in L^{4}(\Omega, \mathbb{R}^{3\times3}) \\
&\operatorname{det}(\nabla y) \in L^{2}(\Omega,\mathbb{R}), \operatorname{det}(\nabla y) > 0 \text{ a.e.} \}
\end{align*}

\begin{rem}
A transformation $y\in \mathscr{A}_0$ is call $y$ an \textbf{admissible transformation}. Note that all admissible transformations fulfill the conditions on the cofactor in Proposition \ref{proposition:degree_boundary} and Theorem \ref{thm:injectivity_closed_constraint}.
\end{rem}

As we see the admissible transformations have strict positive Jacobians a.e. and are thus locally invertible. Note that this deduction is not trivial: Since by the Sobolev embedding theorem \cite[Chapter 5.6.4, Thm 6]{Evans1998} an admissible function does not need to have a continuous version,  we cannot use the implicit function theorem to deduce the existence of a local inversion. Even worse, the standard theory on (local) invertibility of Sobolev mappings is focussed on mappings in $W^{1,d}$ (see for example \cite{Fonseca1995,Kovalev2010}), so this theory would only be applicable for $d=2$. However, we can use some results from the theory of Cartesian currents we mentioned briefly earlier and use the fact, that:
\begin{equation}
\mathscr{A} \subset \mathscr{A}_{d-1,d-1}(\Omega, \mathbb{R}^{m})
\end{equation}

with $\mathscr{A}_{d-1,d-1}(\Omega, \mathbb{R}^{m})$ as being defined in the Appendix. Then a result from M\"uller \cite{Mueller1994} yields the closedness of the graph \cite{Giaquinta1994a} of the transformation $y$ and thus $y$ is a weak local diffeomorphism as defined by Giaquinta et al. \cite{Giaquinta1998}. We will not elaborate on this further, because we used this argumentation only to demonstrate that we can expect to have local invertibility and that this property is not directly guaranteed by the positivity of the Jacobian determinant.

We conclude this short course on hyperelastic regularization by stating the convergence properties, shown in \cite[Chapter 3, Theorem 4]{Ruthotto2012a}:

\begin{theorem}[Convergence properties of admissible transformations ]
\label{thm:properties_hyperelastic_regularization}
Let $\Omega$ be a domain with a $C^1$ boundary and $y_k$, $y \in \mathscr{A}_0$ be admissible transformations. Then convergence of
 
$$y_k \rightharpoonup y \text{ in } W^{1,2}(\Omega,\mathbb{R}^3),  \quad \operatorname{cof} (\nabla y_k) \rightharpoonup H  \text{ in } L^4(\Omega, \mathbb{R}^{3 \times 3}), \quad \operatorname{det}(\nabla y_k) \rightharpoonup v   \text{ in } L^2(\Omega, \mathbb{R})
$$

implies 
$ H=\operatorname{cof}(\nabla y)$ and $v=\operatorname{det}(\nabla y).$

\end{theorem}

\subsection{Existence of a Minimizer}

In this section we establish the following existence result for a compact subset $\mathscr{K}$ of the injective admissible transformations:

\begin{theorem}[Existence of a minimizer in motion-corrected reconstruction]
Let our assumptions (A1)-(A3) hold, furthermore we assume that we have 
\label{theorem:Existence_mcr_functional}
\begin{equation}
\label{eq:mcr_level_set_nonempty}
J(1,(Id)^{N})=\sum\limits_{i=0}^{N}(D(K\rho^0,f^i)+\alpha_i \tv{\rho^0})+\sum\limits_{i=1}^{N}\beta_i S^{hyper}(Id)<\infty,
\end{equation}

Then the functional \eqref{eq:functional}

with a $L^1-$continuous distance term fulfilling coercivity property \eqref{eq:distance_coercivity_condition} has at least one minimizer\newline
$(\hat{\rho},\hat{y}) \in BV(\Omega)\times (\mathscr{K})^N$, where $\mathscr{K}$ is a closed subset of $\mathscr{A}\cap \mathscr{I}$. Particularly this holds for the Kullback-Leibler divergence as distance term. 
\end{theorem}

Injectivity is crucial to enforce the mass-preservation condition. Despite the fact that our proof can be extended to transformations with $N_y(\cdot,\Omega)$ bounded in $L^\infty$ we shortly give two closed subsets of the injective transformations:
\begin{rem}
\label{remark:injectivity_remark}
\begin{itemize}
\item The set $\mathscr{A}\cap \mathscr{I}$ is closed \cite{Henao2010}.
\item The set $\mathscr{B}_v:=\{y \in \mathscr{A} \, : \, y|_{\partial \Omega}=v\}$ is a closed subset of $\mathscr{A}\cap \mathscr{I}$ for an injective boundary value $v$, where the equality is understood in the sense of $H^1$ traces.
\end{itemize}
\end{rem}

Crucial for the proof is controlling sequences of functions after transformations $\rho_k(y_k)\operatorname{det}(\nabla y_k)$ with $\rho_k, y_k$ beeing weak-* convergent sequences resulting from coercivity properties. Before actually presenting our central theorem on this convergence, we define weak-* convergence for the transformation as follows (compare \cite{Ruthotto2012}).

\begin{definition}[Weak-* Convergence in $\mathscr{A}$]
Let $y_k,y \in \mathscr{A}$. We say $y_k \rightharpoonup^* y$ in $\mathscr{A}$, iff
\begin{alignat*}{2}
y_k &\rightharpoonup^* y && \text{ in } W^{1,2}(\Omega,\mathbb{R}^3)\\
\operatorname{cof}(\nabla y_k) &\rightharpoonup^* \operatorname{cof}(\nabla y) \qquad && \text{ in } L^4\left(\Omega, \mathbb{R}^{3\times 3}\right)\\
\operatorname{det}(\nabla y_k) &\rightharpoonup^*  \operatorname{det}(\nabla y) && \text{ in } L^2(\Omega,\mathbb{R})
\end{alignat*}
\end{definition}

Now we can deduce convergence properties of composed weak-* convergent sequences.

\begin{theorem}
\label{theorem:konvergenz}
Let $\rho^{0}_{k} \rightharpoonup^{*} \rho^0 $ in $BV(\Omega_0)$, $y_k \rightharpoonup y$ in $H^{1}(\Omega)$ and $\operatorname{det}(\nabla y_k) \rightharpoonup \operatorname{det}(\nabla y)$ in $L^2(\Omega)$. Assume additionally, that
\begin{equation}
\sup\limits_{k} \|N_{y_k}(\cdot,\Omega)\|_{\infty} \leq C \in \mathbb{R}.
\label{eq:boundindicatrix}
\end{equation}
Then we obtain $\rho^0_k(y_k)\operatorname{det}(\nabla y_k) \rightharpoonup \rho^0(y)\operatorname{det}(\nabla y) \text{ in } L^{1}(\Omega)$.

\end{theorem}

\begin{proof}

For the ease of presentation we use the following abbreviation:

\begin{equation*}
\rho_k:=\rho^0_k \quad \rho:=\rho^0 \quad d_k:=\operatorname{det}(\nabla y_k)
\end{equation*}

For any fixed $\varphi \in (L^{1})^{*}=L^{\infty}$ we show:
\begin{align*}
0 = \lim_{k \rightarrow \infty}&\int\limits_{\Omega}{(\rho_k(y_k)d_k - \rho(y)d)\varphi \mathrm{d}x} \\
=\lim_{k \rightarrow \infty}&\int\limits_{\Omega}(\rho_k(y_k)d_k - \rho(y_k)d_k + \rho(y_k)d_k - \rho(y)d) \varphi \mathrm{d}x \\
= \lim_{k \rightarrow \infty} &\int\limits_{\Omega}{(\rho_k(y_k)d_k - \rho(y_k)d_k)\varphi \mathrm{d}x} + \lim_{k \rightarrow \infty}\int\limits_{\Omega}{(\rho(y_k)d_k - \rho(y)d)\varphi \mathrm{d}x} 
\end{align*}
We examine both summands separately and show that each of them converges to zero. 
We recall that weak star convergence in $BV$ implies strong convergence in $L^{1}$ \cite{Burger2013a}. Now the first term is treated in a straightforward way by the area formula and  \eqref{eq:boundindicatrix}
\begin{align*}
\lim_{k \rightarrow \infty}\int\limits_{\Omega}{\left|\rho_k(y_k)d_k - \rho(y_k)d_k\right|\mathrm{d}x} 
&=\lim_{k \rightarrow \infty}\int\limits_{\mathbb{R}^{n}}{\left|\rho_{k}-\rho\right|N_{y_{k}}(x,\Omega)\mathrm{d}x}\\
&\leq \lim_{k \rightarrow \infty} C\int\limits_{\mathbb{R}^{n}}{\left|\rho_{k}-\rho\right|\mathrm{d}x}
\end{align*}
Since  $\operatorname{supp}(\rho_k -\rho) \subseteq \Omega_0$ and $\rho_{k} \rightarrow \rho$ in $L^{1}(\Omega_0)$ implies weak convergence, we obtain
\begin{equation}
\label{eq:bigsummand1}
\lim_{k \rightarrow \infty}\int\limits_{\Omega}{(\rho_k(y_k)d_k - \rho(y_k)d_k)\varphi \mathrm{d}x}=0 \hspace{0.5cm} \forall \varphi.
\end{equation}
The second term needs further care: According to \cite[Thm. 3.9]{Ambrosio2000}. we find a sequence of functions $(\xi^{n})_n \subset C^{\infty}(\Omega_0)$ with $\lim\limits_{n\rightarrow \infty}\|\xi^{n} - \rho\|_1=0$.

Let $\epsilon > 0$. Thus we can pick a fix N, such that:
\begin{equation}
\int\limits_{\Omega_0}{|\xi^{N}-\rho| \mathrm{d}x}\leq \frac{\epsilon}{4\|\varphi\|_{\infty}C}. 
\label{eq:glatteapprox}
\end{equation}
We now expand the first term with said $\xi^{N}$ and obtain:
\begin{align*}
&\int\limits_{\Omega}(\rho(y_k)d_k- \xi^{N}(y_k)d_k + \xi^{N}(y_k)d_k - \xi^{N}(y)d_k \\
&+\hphantom{\int\limits_{\Omega}} \xi^{N}(y)d_k - \xi^{N}(y)d + \xi^{N}(y)d - \rho(y)d)\varphi \mathrm{d}x \\
& \quad=\int\limits_{\Omega}{(\rho(y_k)d_k- \xi^{N}(y_k)d_k)\varphi \mathrm{d}x} + \int\limits_{\Omega}{(\xi^{N}(y_k)d_k - \xi^{N}(y)d_k)\varphi \mathrm{d}x} \\
&+ \quad\int\limits_{\Omega}{(\xi^{N}(y)d_k - \xi^{N}(y)d)\varphi \mathrm{d}x} + \int\limits_{\Omega}{(\xi^{N}(y)d - \rho(y)d)\varphi \mathrm{d}x },\\
\end{align*}
We examine each term separately: We start with applying H\"older's inequality to the first one and obtain
\begin{align*}
 \int\limits_{\Omega}{(\rho(y_k)d_k- \xi^{N}(y_k)d_k)\varphi \mathrm{d}x} 
&\leq \int\limits_{\Omega}{|\rho(y_k)d_k- \xi^{N}(y_k)d_k|\mathrm{d}x} ~ \|\varphi\|_{\infty} \\
&=\int\limits_{\mathbb{R}^{n}}{|\rho - \xi^{N}|N_{y_k}(z,\Omega)\mathrm{d}z} \|\varphi\|_{\infty} \\
&\leq C \|\rho - \xi^{N}\|_1  \|\varphi\|_{\infty},
\end{align*}
which yields together with \eqref{eq:glatteapprox} 
\begin{equation}
\int\limits_{\Omega}{(\rho(y_k)d_k- \xi^{N}(y_k)d_k)\varphi \mathrm{d}x} \leq \frac{\epsilon}{4} .
\label{eq:summand1}
\end{equation}
In the second term we start in a straightforward way with
\begin{equation}
\int\limits_{\Omega}{|\xi^{N}(y_k)d_k|\mathrm{d}x} \leq \|\xi^{N}(y_k)\|_{2}\|d_k\|_2.
\label{eq:erstesintegral}
\end{equation}
By using that $\xi^{N}(y_k)$ is bounded, $\Omega$  compact and $y_k$ is weakly convergent, it directly follows, that $\|\xi^{N}(y_k)\|_{2}<\infty$ and $\|d_k\|_2$ is bounded by some $E\in\mathbb{R}$ (Banach-Steinhaus, see e.g. \cite[Chapter 5, Theorem 3]{Alt2002}). Thus \eqref{eq:erstesintegral} is finite. Aswell we can directly deduce
\begin{equation}
\int\limits_{\Omega}{|\xi^{N}(y)d_k|\mathrm{d}x} \leq \|\xi^{N}(y)\|_{2}\|d_k\|_2 \leq \infty,
\label{eq:zweitesintegral}
\end{equation}
and hence,
\begin{align*}
\int\limits_{\Omega}{(\xi^{N}(y_k)d_k - \xi^{N}(y)d_k)\varphi \mathrm{d}x} 
&\leq  \|(\xi^{N}(y_k)- \xi^{N}(y))\varphi\|_{2}\|d_k\|_{2} \\
&\leq   \|(\xi^{N}(y_k)- \xi^{N}(y))\varphi\|_{2} ~E 
\end{align*}
Note that $\xi^{N}$ is in $C^{\infty}$ and therefore $\left(\xi^{N}\right)^{2}$ is Lipschitz-continuous with some constant L, since $\Omega$ is compact. We obtain
\begin{equation}
\int\limits_{\Omega}{\left((\xi^{N}(y_k) - \xi^{N}(y))\varphi\right)^{2} \mathrm{d}x} \leq \left(\|\varphi\|_{\infty}\right)^{2} L \|y_k-y\|_{2}.
\end{equation}
By the Rellich-Kondrachov compactness theorem \cite[Chapter 5.7, Theorem 1]{Evans1998} $y_k$ converges strongly to $y$ in $L^{2}$ and so we find $K_2 \in \mathbb{N}$, such that
\begin{equation}
\|y_k-y\| < \frac{\epsilon^{2}}{16 \left(\|\varphi\|_{\infty}\right)^{2} L E^{2}}. \qquad \forall k\geq K_2
\end{equation}
This implies for each $k \geq K_2$
\begin{equation}
\int\limits_{\Omega}{(\xi^{N}(y_k)d - \xi^{N}(y)d)\varphi \mathrm{d}x} < \frac{\epsilon}{4}.
\label{eq:summand2}
\end{equation} 
Now let us considered the third term. 
Since $\xi^{N}(y_k)$ is bounded it follows from the weak convergence of the determinants, that there exists $K_3 \in \mathbb{N}$, such that for every $k \geq K_3$ 
\begin{equation}
\label{eq:summand3}
\int\limits_{\Omega}{\left(\xi^{N}(y)d_k - \xi^{N}(y)d\right)\varphi \mathrm{d}x} \leq \frac{\epsilon}{4}. 
\end{equation}
Note that the compactness of $\Omega$ ensures that $\varphi \xi^{N}(y) \in L^{2}$. 
For the last term we can proceed as for the first one, with $K_1$ as above, and deduce
\begin{equation}
\label{eq:summand4}
\int\limits_{\Omega}{\left(\xi^{N}(y)d - \rho(y)d\right)\varphi \mathrm{d}x } \leq \frac{\epsilon}{4} 
\end{equation}
 for any $k \geq K_1$. 

By combining \eqref{eq:summand1}, \eqref{eq:summand2}, \eqref{eq:summand3} and  \eqref{eq:summand4}, we obtain for every 
$k \geq K := \max \{K_1,K_2,K_3\}$
\begin{equation*}
\int\limits_{\Omega}{(\rho(y_k)d_k - \rho(y)d)\varphi \mathrm{d}x}\leq \epsilon \hspace{0.5cm} \forall \varphi
\end{equation*}
and thus
\begin{equation}
\label{eq:bigsummand2}
\lim\limits_{k \rightarrow \infty}\int\limits_{\Omega}{(\rho_k(y_k)d_k - \rho_k(y)d)\varphi \mathrm{d}x}= 0.
\end{equation}
The assertion follows by combining \eqref{eq:bigsummand1} and \eqref{eq:bigsummand2}.
\end{proof}

As we see, the boundedness of the Banach-Indicatrix is substantial to control sequences of images after transformations. In order to guarantee the convergence properties given by Theorem \ref{theorem:konvergenz}, we restrict our analysis to injective transformations. Additionally this has the effect that as a consequence of the area formula the mass-preservation condition is not violated.

Now we can give continuity properties for a wide range of distance measures including the Kullback-Leibler divergence:
\begin{lemma}
\label{lem:distance_semicontinuous}
Let the assumptions from Theorem \ref{theorem:konvergenz} be fulfilled. Then the distance part of our functional J, defined by
\begin{equation}
D(\rho_0,y)=\sum\limits_{i=0}^N\int\limits_{\Sigma}{g(K\rho^0(y^i)\operatorname{det}(\nabla y^i),f^i)\mathrm{d}\sigma}
\end{equation}
with a $L^1$-continuous integrand function $g$ is lower semicontinous with respect to weak-star convergence in $BV(\Omega_0)$, weak convergence in $W^{1,2}(\Omega_i)$ for the transformations and weak convergence in $L^2(\Omega_i)$ for the determinants.
\end{lemma}
\begin{proof}
Let $y_k \rightharpoonup y$ in $W^{1,2}(\Omega)$, $\operatorname{det}(\nabla y_k) \rightharpoonup \operatorname{det}(\nabla y)$ in $L^2$, $\rho^{0}_{k} \rightharpoonup \rho^0$ in $BV(\Omega)$.
Since we denote $y=(y^1,...y^N)$ as the collection of all transformations we understand the weak convergence componentwise. 
From Theorem \ref{theorem:konvergenz} we obtain for any fixed $1\leq i \leq N$ 
\begin{equation*}\rho^{0}_{k}(y_k^i)det(\nabla y_k^i) \rightharpoonup \rho^0(y^i)det(\nabla y^i). 
\end{equation*}
Since $K$ is a compact operator, K is completely continuous, which gives us:
\begin{equation*}
K (\rho^{0}_{k}(y_k^i)det(\nabla y_k^i)) \rightarrow K (\rho^0(y^i)\operatorname{det}(\nabla y^i)) \qquad \text{ in } L^{1}(\Sigma).
\end{equation*} 
Therefore we can follow the proof of Lemma 3.4 (iii) from \cite{resmerita/anderssen} for each summand and obtain lower semicontinuity of the Kullback-Leibler data fidelity term with the following reasoning:

Since $K (\rho^{0}_{k}(y_k^i)det(\nabla y_k^i)) \rightarrow K (\rho^0(y^i)\operatorname{det}(\nabla y^i))$, we have convergence almost everywhere. Thus we can deduce
\begin{equation*}
g(K\rho^{0}_{k}(y_k^i)\operatorname{det}(\nabla y_k^i))\rightarrow g(K\rho^{0}(y^i)\operatorname{det}(\nabla y^i)) \qquad \text{ in } L^1(\Sigma).
\end{equation*}
Now we can apply Fatou's Lemma \cite[Chapter 1, Theorem 2 (iii)]{Giaquinta1998a} and obtain:
\begin{equation*}
\int\limits_{\Sigma}g(K\rho^{0}(y^i)\operatorname{det}(\nabla y^i))\mathrm{d}\sigma\leq \liminf\limits_{k\rightarrow \infty}\int\limits_{\Sigma}g(K\rho^{0}_{k}(y_k^i)\operatorname{det}(\nabla y_k^i))\mathrm{d}\sigma
\end{equation*}
Having shown lower semicontinuity for an arbitrary summand with fixed i, the assertion follows directly.
\end{proof}
%

We have now stated lower semicontinuity results for a wide range of distance terms, including the Kullback-Leibler data fidelity. We now turn our focus to the TV-regulari\-za\-ti\-on. By setting $\alpha^k=0$ for any $1\leq k \leq N$ in \eqref{eq:functional}, this would follow directly by the properties of the BV-seminorm we mentioned earlier. However to formulate an existence result for $\alpha^k geq 0$, we give the following lemma, which follows with a proof as in  \cite{Burger2013a}:
\begin{lemma}
\label{lem:tv_semicontinuous}
Let $\rho_k \rightharpoonup \rho$ in $L^{1}(\Omega)$. Then 
\begin{equation}
\tv{\rho} \leq  \liminf\limits_{k\rightarrow \infty} \tv{\rho_k}.
\end{equation}
\end{lemma}

We verified the first condition for the existence of a minimizer and turn our focus coercivity.

\begin{lemma}[$\rho$-Coercivity for the Kullback-Leibler divergence]
\label{lem:distance_coercive}
Let our assumptions \eqref{eq:hyperelasmenge} for our transformation and (A1)-(A3) for the operator be fulfilled. Then 
\begin{align*}
J^{1}(\rho^0,y):=\sum\limits_{i}&\left[{\int\limits_{\Sigma}{K (\rho^0 (y^i det(\nabla y^i)))-f^i\log(K (\rho^0 (y^i det(\nabla y^i))))\mathrm{d}\sigma}}\right.\\
&\left.+\alpha^i\tv{\rho^0(y^i)\operatorname{det}(\nabla y^i)}\right]
\end{align*}
is coercive with respect to the variable $\rho^0$.
\end{lemma}
\begin{proof}
We begin by observing:
\begin{equation*}
\log(x)\leq \frac{1}{a}x+\log(a) \hspace{1cm} \forall x>0
\end{equation*}
for any fixed $a>0$.  We add the constant $f \log(f) -f$, such that each summand is non-negative. Note that $y^0=Id$ and therefore we can bind the Kullback-Leibler divergence from below by:
\begin{align*}
&\sum\limits_{i=0}^{N}{\int\limits_{\Sigma}{K(\rho^0(y^i)det(\nabla y^i)) - f^i \log(K(\rho^0(y^i)det(\nabla y^i))) + f^i \log(f^i) -f^i}\mathrm{d}\sigma} \\
\geq&\int\limits_{\Sigma}{K(\rho^0(y^0)det(\nabla y^0))-f^0\log(K(\rho^0(y^0)det(\nabla y^0))) + f^0 \log(f^0) -f^0\mathrm{d}\sigma}\\
=&\int\limits_{\Sigma}{K\rho^0-f^0\log(K\rho^0) + f^0\log(f^0) -f^0\mathrm{d}\sigma}\\
\geq&\int\limits_{\Sigma}{K\rho^0-f^0\left(\frac{1}{f^0+1}K\rho^0 +\log(f^0+1) \right)+ f^0\log(f^0) -f^0 \mathrm{d}\sigma} \\
\geq &\int\limits_{\Sigma}{\frac{1}{f^0+1}K\rho^0 \mathrm{d}\sigma}\underbrace{-\int\limits_{\Sigma}{f^0(f^0+1)+ f^0\log(f^0)-f^0 \mathrm{d}\sigma}}_{:=c_2\in \mathbb{R}}\\
=&\int\limits_{\Omega}{K^{*} \frac{1}{f^0+1} \rho^0 \mathrm{d}x}+c_2 
\geq \int\limits_{\Omega}{c_1 \rho^0 \mathrm{d}x}+c_2\mathrm{d}x \\
=&c_1\|\rho^0\|_1 +c_2
\end{align*}
Therefore we can conclude
\begin{equation}
J_1(\rho^0,y)\geq c_1 \|\rho^0\|_1+\alpha^0\tv{\rho^0}\geq\min\{c_1,\alpha^0\}(\|\rho^0\|_1+\tv{\rho^0}).
\end{equation}
\end{proof}

\begin{rem}
The lemma holds not only for the Kullback-Leibler divergence, but for any distance of the form
\begin{equation}
D(K\rho^i,f^i)=\int\limits_{\Sigma}g^i(K\rho^i)\mathrm{d}\sigma
\end{equation}´
with $g^i$ convex, which is bounded below and satisfies
\begin{equation}
\label{eq:distance_coercivity_condition}
g^0(K\rho^0)\geq c_1 K\rho^0+c_2
\end{equation}
with constants $c_1 >0$ and $c_2\in \mathbb{R}$, only dependent of $f$.
\end{rem}

We have proved that the first part of our functional is coerciv in $\rho_0$, so it remains to be shown, that the hyperelastic regularization is coercive with respect to $y$. Fortunately this has already been done in \cite[Chapter 2, Lemma 1]{Ruthotto2012a}. Combining these results, we can finally prove Theorem \ref{theorem:Existence_mcr_functional}:
\begin{proof}(of Theorem \ref{theorem:Existence_mcr_functional})
The coercivity properties of distance \eqref{eq:distance_coercivity_condition}, TV and hyperelastic regularization ensure the existence of a bounded level set. Compactness is then granted by the  Banach-Alaoglu Theorem. Thus a minimizing sequence $(\rho_k,y_k)$ has a convergent subsequence with limit $(\rho,y)$, while the convergence is understood component wise. By the lower semicontinuity properties of distance (Lemma \ref{lem:distance_semicontinuous}), TV (Lemma \ref{lem:tv_semicontinuous}) and hyperelastic regularization we obtain:
\begin{equation*}
J(\rho,y) \leq \liminf J(\rho_k,y_k) = \inf J
\end{equation*}
Finally as a result of the compactness of $\mathscr{K}$ we have $(\rho,y) \in BV(\Omega) \times \mathscr{K}^{N}$.
\end{proof}

\subsection{Convergence of the Regularization Method}

Having stated existence results for the variational problem of motion-corrected reconstruction we can show with analogous weak compactness and lower semicontinuity arguments that the minimization of our functional \eqref{eq:functional} can be understood as a (nonlinear) regularization method, i.e. there exist appropriate limites as noise and regularization parameter tend to zero \cite{Engl1996}. 
\begin{theorem}
\label{thm:mcr_reg_method}
Let $(f^i_k)_k$ be a sequence of noisy data with
\begin{equation}
\lim_k f^i_k = f^i_{*},
\end{equation}
where $f^{*}$ is the exact data for an image $\rho_*$ and transformations $y_*$, such that:
\begin{equation}
(\rho_*,y_*)=\min\limits_{\rho,y} \sum\limits_{i=0}^{N}D(K(\rho(y^ i)\operatorname{det}(\nabla y^i )),f^i_*)
\end{equation}
for a non-negative distance $D$, which is lower semicontinuous in both arguments and fulfilling \eqref{eq:distance_coercivity_condition}. Furthermore, we define a sequence of functionals $J_k$ by:
\begin{equation}
J_k=\sum\limits_{i=0}^{N}(D(K(\rho^0(y^i)\operatorname{det}(\nabla y^i)),f^i_k)+\alpha^i_k \tv{\rho^i})+\sum\limits_{i=1}^{N}\beta^i_k S^{hyper}(y^i)
\end{equation}
Then for $\alpha_k \rightarrow 0$ and $\beta_k \rightarrow 0$   with
\begin{equation}
\label{eq:mcr_parameter_choice}
\frac{\sum\limits_{i=0}^{N}D(f_*^i,f^i_k)}{\min\limits_{i} \{a_k^i,\beta_k^i\}} \rightarrow 0 \qquad  \frac{\max\limits_{i} \{a_k^i,\beta_k^i\}}{\min\limits_{i} \{a_k^i,\beta_k^i\}} \leq C \in \mathbb{R} \forall k
\end{equation} 
the sequence $(\hat{\rho}_k,\hat{y}_k)_k$, with $(\hat{\rho}_k,\hat{y}_k)$ being minimizers of $J_k$ has a convergent subsequence and the limit $(\hat{\rho},\hat{y})$ fulfills: 
\begin{equation}
\sum\limits_{i} D(\hat{\rho}(\hat{y}^i)\operatorname{det}(\nabla \hat{y}^i),f_*^i)= \sum\limits_{i} D(\rho_*(y_*^i)\operatorname{det}(\nabla y_*^i),f_*^i)
\label{eq:estimation_compositon}
\end{equation}
\end{theorem}

Note that the estimation \eqref{eq:estimation_compositon} only holds for the composition and not for the components. Since different transformations can lead to the same transformed image  we can in general not expect to derive convergence for both components. 

\begin{rem}
In order to deal with the non uniqueness of the solution yielded by Theorem \ref{thm:mcr_reg_method} we assume additionally
\begin{enumerate}
\item $ \lim\limits_{k}\frac{\sum\limits_{i=0}^{N}D(K\rho^i_*,f^i_k)}{\min\limits_{i}\{\alpha_k^i,\beta_k^i\}}=0$ ,
\item For all $\alpha^i_k$, $\beta^i_k$ there exist $\tilde{\alpha}^i=\lim\limits_{k}\frac{\alpha^i} {\min\limits_{i}\{\alpha_k^i,\beta_k^i\}}$, resp. $\tilde{\beta}^i=\lim\limits_{k}\frac{\beta^{i}}{\min\limits_{i}\{\alpha_k^i,\beta_k^i\}}$,
\end{enumerate}
then we can deduce
\begin{align*}
&\sum\limits_{i=0}^{N}\tilde{\alpha}^i\tv{\hat{\rho}^i}+\sum\limits_{i=1}\tilde{\beta}^i S^{hyper}(\hat{y}^i)\\
\leq \liminf\limits_{k}&\sum\limits_{i=0}^{N}\tilde{\alpha}^i\tv{\hat{\rho}^i_k}+\sum\limits_{i=1}^N\tilde{\beta}^iS^{hyper}(\hat{y}^i_k)\\
\leq\lim\limits_{k}&\left(\frac{\sum\limits_{i=0}^{N}D(K\rho^i_*,f^i_k)}{\min\limits_{i}\{\alpha_k^i,\beta_k^i\}}+\sum\limits_{i=0}^{N}\tilde{\alpha}^i\tv{\rho^i_*}+\sum\limits_{i=1}^N\tilde{\beta}^iS^{hyper}(y^i_*)\right)\\
=&\sum\limits_{i=0}^{N}\tilde{\alpha}^i\tv{\rho^i_*}+\sum\limits_{i=1}^N\tilde{\beta}^iS^{hyper}(y^i_*).
\end{align*}
By using this deduction we can show that $(\hat{\rho},\hat{y})$ is a solution, which minimizes 
\begin{equation*}
\sum\limits_{i=0}^N \tilde{\alpha}^i\tv{\cdot}+\sum\limits_{i=1}^N\tilde{\beta}^iS^{hyper}(\cdot) \quad \forall \quad (\rho,y) \quad \text{with} \quad \sum\limits_{i=0}^{N}D(K\rho^i,f_*^i)=0.
\end{equation*}
This solution can be viewed analogously to the best-approximate solution in the sense of Engl et al. \cite[Definition 2.1]{Engl1996}.
\end{rem}

\section{Numerical Solution}
In this section we describe the numerical framework we use to solve the motion-corrected reconstruction problem. We restrict the presentation of the numerical framework to the Kullback-Leibler divergence as distance measure; nevertheless the extension to other distances with given TV-regularized reconstruction algorithms is straightforward. 

We aim to perform a First-Discretize-then-Optimize approach combined with an alternating minimization strategy. For this we only need a discretization in the space-domain, since we assume to have time-discretized data. For the time-discretization we assume, that the discrete data $(f_i)_i$ is gated: Therefore we define time nodes $(t_i)_i$ such that each node $t_i$ represents a stage of motion (f.e. cardiac or respiratoric gate). Additionally we impose regularization only on the reference configuration $\rho_0$ and not on the transformed versions of $\rho_0$, so by setting $\alpha^i=0$ for $i>0$ we obtain the following functional to be minimized

\begin{align*}
J(\rho_0,y)&=\sum\limits_{i}{\int\limits_{\Sigma}{K (\rho^0 (y^i) det(\nabla y^i))-f^i\log(K (\rho^0 (y^i( det(\nabla y^i)))))\mathrm{d}\sigma}} \\
&+ \tv{\rho^0}+\sum\limits_{i}{S^{hyper}(y^i)} .\\
\end{align*}

Out discretization is straightforward: We define pairwise disjoint pixel $B_i$, such that

\begin{equation}
\Omega = \bigcup_i B_i.
\end{equation}

With this discretization at hand we can put our problem in a discrete framework and start to minimize our functional. Given an initial value $\rho_0^{0}$ our algorithm reads as follows:
\begin{equation*}
\begin{cases}
\text{1. Motion step:} & y_{k+1}\in\underset{y}{\argmin}  \left\{J(\rho^0_{k},y)\right\} \\
\text{2.  Reconstruction step:} & \rho^0_{k+1}\in\underset{\rho^0}\argmin \left\{ J(\rho^0,y_{k+1}) \right\}
\end{cases}.
\end{equation*}

While the reconstruction step can be realized via motion corrected EM-TV algorithms (see e.g. \cite{Brune2009}), the motion step needs some caretaking. In the next sections we want to outline both implementations briefly.

\subsection{Reconstruction-Step: Motion-Corrected EM-TV}
\label{subsection:reconstruction_step}

In this section we want to show that we can apply standard EM-TV algorithms to motion-corrected reconstructions. For ease of presentation we assume that the transformations $y^i$ are global invertible, as e.g. in the case of Dirichlet boundary conditions (Proposition \ref{proposition:degree_boundary}). The central idea for EM-TV is to alternate an EM step with given discrete projection operator (so called system matrix) with a TV denoising step \cite{Sawatzky2013}. For a classical reconstruction problem
\begin{equation}
\label{eq:recon_tv_standard}
\min\limits_{\rho} \int\limits_{\Sigma}K\rho - f \log(K\rho)+\alpha \tv{\rho}
\end{equation}
 the FB-EMTV algorithm is given by alternation of a reconstruction (EM step) with a denoising (TV) step:
\begin{equation*}
\begin{cases}
u_{k+\frac{1}{2}}=\frac{u_k}{K^* 1_{\Sigma}}K^*\left(\frac{f}{Ku_{k}}\right) \\
u_{k+1}\in \argmin\limits_{u \in BV(\Omega)}\left \{\frac{1}{2}\int\limits_{\Omega}\frac{K^* 1_{\Sigma}\left(u-u_{k+\frac{1}{2}}\right)^2}{u_k}\mathrm{d}x+\alpha \tv{u}\right\}
\end{cases}
\end{equation*}

We will show how to modify the distance term, so that it fits the structure described in \eqref{eq:recon_tv_standard}. Since the transformation is a linear operator acting on the image,  we can write our transformation model \eqref{eq:transformation_model} equivalently as

\begin{equation}
\rho^{i}=T_{y^i}^{mp} \rho_0.
\end{equation} 

Projecting the transformed reference configuration $\rho^i$ into the measurement domain $\Sigma$ is then given by

\begin{equation}
A^{0}\rho^i=A^{0}\left(T_{y^i}^{mp} \rho^0\right).
\end{equation}

Now using the associativity property leads to
\begin{equation}
A^{0}\rho^{i}=\underbrace{\left(A^0 T_{y^i}^{mp}\right)}_{=:A^i}\rho^0,
\end{equation}
which gives us a motion-corrected projection operator $A^i$, acting on $\rho^0$ only. Now we can formulate the motion-corrected reconstruction problem as
\begin{equation}
\label{eq:motioncorrrec}
\min\limits_{\rho} \int\limits_{\Sigma}\underbrace{\begin{pmatrix} A^{0} \\ A^{1} \\ \cdots \\ A^{N} \end{pmatrix}}_{:=K}\rho^0- \underbrace{\begin{pmatrix} f^0 \\ f^1 \\ \cdots \\f^{N}\end{pmatrix}}_{:=f}\log(A\rho^0)+\alpha \tv{\rho^0}
\end{equation}
After this tuning our problem is suited for the the EM-TV algorithm \cite{Sawatzky2013} with $u:=\rho^0$.

\subsection{Motion-Step: Interpretation as Registration}

In the Motion-Step we need to minimize
\begin{align*}
&\sum\limits_{i}{\int\limits_{\Sigma}{K (\rho_0 (y^i) \operatorname{det}(\nabla y^i))-f_i\log(K (\rho_0 (y^i)\operatorname{det}(\nabla y^i)))\mathrm{d}\sigma}} \\
&+ \tv{\rho_0}+\sum\limits_{i}{S^{hyper}(y^i)} \\
\end{align*}

with respect to our set of transformations $(y^i)_i$. Note that the problem decouples, so we obtain a minimum of the sum, by minimizing each summand. Since $\tv{\rho_0}$ is a constant we retain a problem of the form:
\begin{equation}
\min\limits_{y^i}\int\limits_{\Sigma}{K (\rho_0 (y^i) det(\nabla y))-f_i\log(K (\rho_0 (y^i) \operatorname{det}(\nabla y)))\mathrm{d}\sigma}+S^{hyper}(y^i)
\label{eq:mcr_registration_problem}
\end{equation}

If we consider the Kullback-Leibler divergence being an distance measure, this is the form of a standrad registration problem from \cite{Modersitzki2009} with hyperelastic regularization. Nevertheless we should mention that the Kullback-Leibler distance measure is defined on the detector domain $\Sigma$ and not on the image domain $\Omega$ like standard distance measures such as SSD or the Normalized Gradient Field \cite{Modersitzki2009}. Similar to the widely known FAIR toolbox \cite{Modersitzki2009} we perform a quasi Newton type optimization, which we will outline shortly in the following, starting with the discretization:

Since $\Sigma$ is the detector domain and therefore discrete with size $m_{\Sigma}$, the integral becomes a sum:
\begin{equation}
\sum\limits_{j=1}^{m_{\Sigma}}{K (\rho_0 (y^i) \operatorname{det}(\nabla y^i))(j)-f_i(j)\log(K (\rho_0 (y^i) det(\nabla y^i))(j))}+S^{hyper}(y(\cdot,t_i))
\label{eq:disregistrierung}
\end{equation}
Rather than classical Gauss-Newton we aim to perform a minimization with a modified BFGS method \cite{Dong-HuiLi2001}. In order to minimize the objective function \eqref{eq:disregistrierung} we need to compute the derivatives with respect to the transformation grid. We assume that the transformation is given on a nodal grid:

\begin{equation}
yc\in \mathbb{R}^{\tilde{m}} \qquad \tilde{m}=d\prod\limits_{i=1}^{d}(m_i+1) ,
\end{equation}

where $m=(m_1,..,m_d)$ denotes the size of the digital image obtained in the reconstruction step. Again we use the interpolation $\operatorname{inter}$ and computation of the Jacobian determinant $\operatorname{jac}$, implemented in the FAIR toolbox \cite{Modersitzki2009}. Then for any continuously differentiable distance term 
\begin{equation} 
D:\left(\mathbb{R}^{m_{\Sigma}}\right)^{2} \rightarrow \mathbb{R} \qquad (K (\operatorname{inter}(\rho^0,yc)\operatorname{jac}(yc)),f) \rightarrow D(K(\operatorname{inter}(\rho^0,yc)\operatorname{jac}(yc)),f)
\end{equation}
the derivative with respect to the transformation grid $yc$ is given by the chain rule as
\begin{align}
&\frac{\mathrm{d}}{\mathrm{d}\, yc}D(K(\operatorname{inter}(\rho^0,yc)\operatorname{jac}(yc)),f)\nonumber\\
=&\frac{\mathrm{d}}{\mathrm{d} \, w}D(K(\operatorname{inter}(\rho^0,yc)\operatorname{jac}(yc)),f)\left(K\left(\frac{\mathrm{d}}{\mathrm{d} \, yc}\operatorname{inter}+\frac{\mathrm{d}}{\mathrm{d} \, yc} \operatorname{jac}\right)\right).
\label{eq:projected_distance_derivative}
\end{align}
With the help of \eqref{eq:projected_distance_derivative} we can deal with different distance terms in the same objective function. 

The actual registration is then performed with a multilevel approach. Since the field of view in the scanner is often bigger than the studied object, it is reasonable to impose Dirichlet boundary conditions on the motion with $y|_{\partial \Omega}=Id$, which guarantees the existence of a minimizer (Remark \ref{remark:injectivity_remark}). This boundary conditions can be realized by taking the identity as starting guess and modifying the search directions to zero at the boundary. 

\section{Results}

This section is divided into two parts. In the first part we study the performance of motion and image estimation with help of a simple artificial deblurring problem. After having shown the superiority of the proposed method on this dataset we turn our focus towards a proof of concept for the clinical applicability by studying the XCAT software phantom.

\subsection{Artificial Deblurring Example}

Deblurring problems are often occuring in microscopy, where the exact image gets convoluted with an (often unknown) point-spread function \cite{Landi2012,Carlavan2012}. Since the focus in this section lies on the motion-corrected reconstruction, we  assume the exact blur operator to be known. In order to assess the performance of the proposed method we consider a ring shaped object in three different stages of shrinkage (Figure \ref{fig:noisy_data}).

\begin{figure}[H]
\label{fig:noisy_and_error}
\centering
\subcaptionbox{Noisy Data\label{fig:noisy_data}}{\includegraphics[scale=0.3]{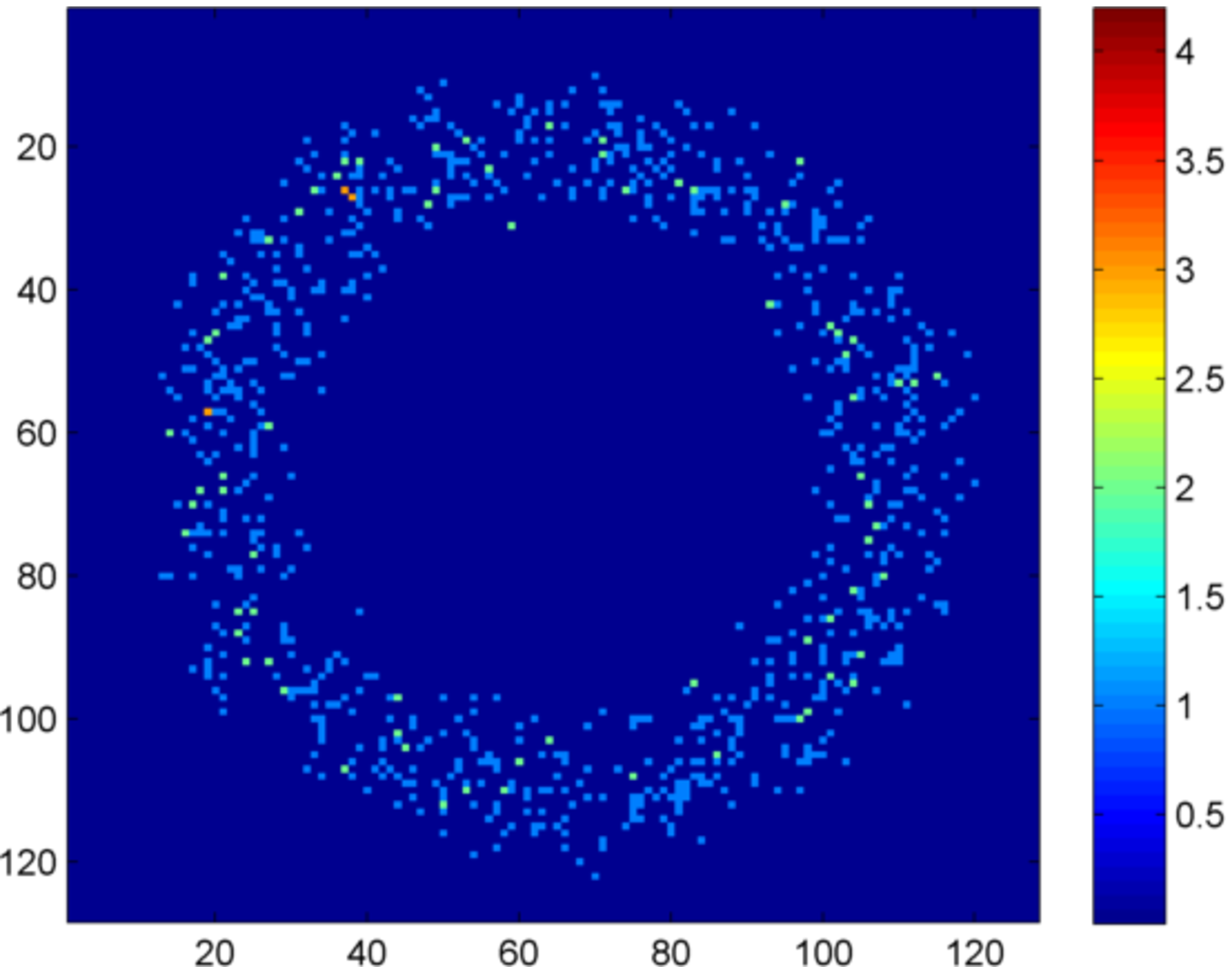}\hspace{0.2cm}\hspace{0.2cm}}
\subcaptionbox{Reconstruction Error\label{fig:recon_error}}{\includegraphics[scale=0.35]{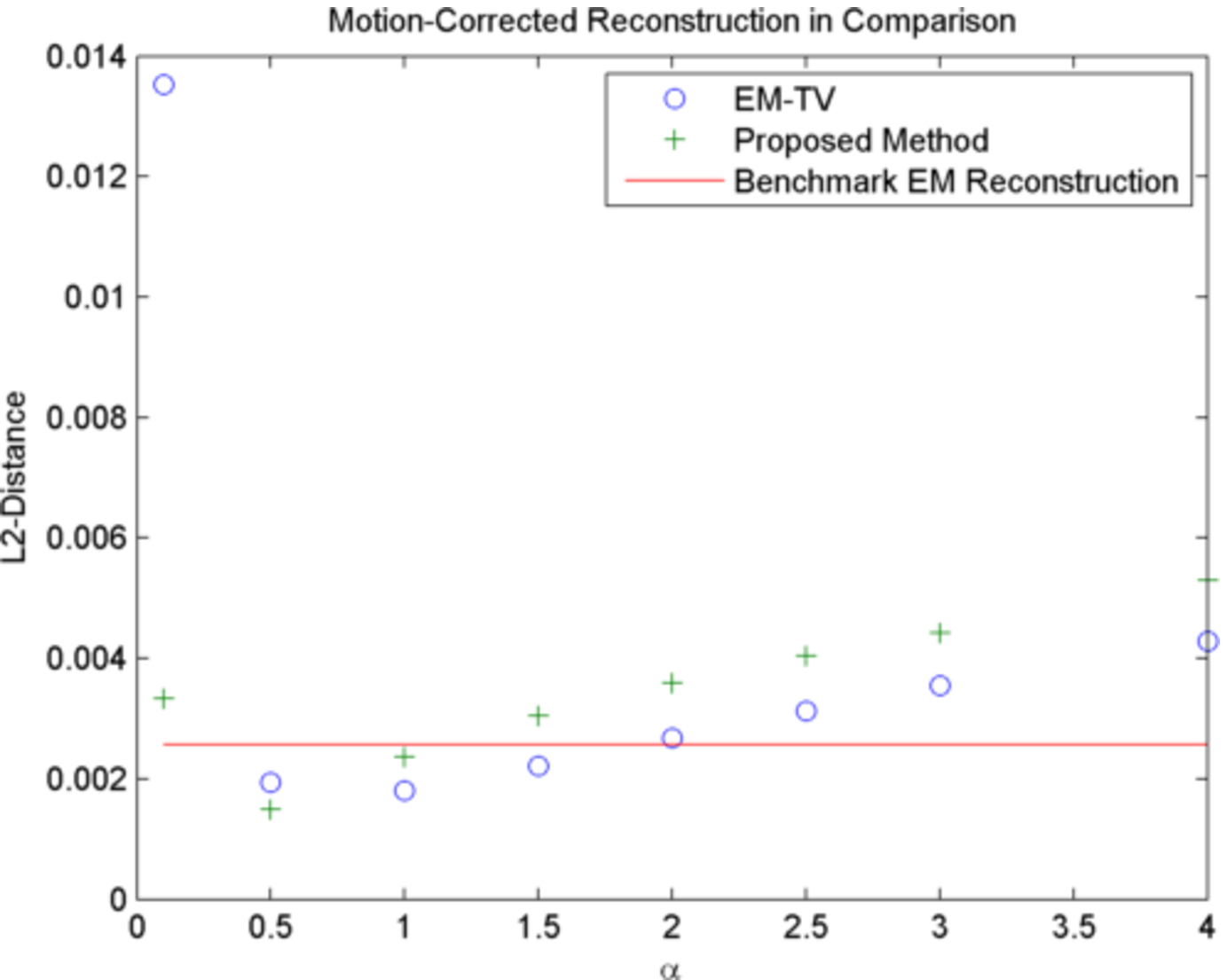}}
\caption{Noisy data for one Gate and reconstruction errors or the proposed Method ($5$ Bregman Iterations), EM-TV ($5$ Bregman Iterations) reconstruction and reconstruction error yielded by classical EM as benchmark. The proposed method yields the best reconstruction result. }
\end{figure}

As we see in Figure \ref{fig:recon_error} the proposed method performs better than the TV-regularized expectation maximization. To conclude this quantitative comparison we focus on different motion estimation methods. In order to assess the quality of the estimation we compare the motion estimation by the proposed method with an affine mass-preserving 2D registration performed on TV regularized single gate reconstructions. For doing so, we picked the best reconstruction in the terms of the reconstruction error for the EM-TV reconstruction of each gate and registered them.

\begin{figure}[H]
\centering
\includegraphics[scale=0.35]{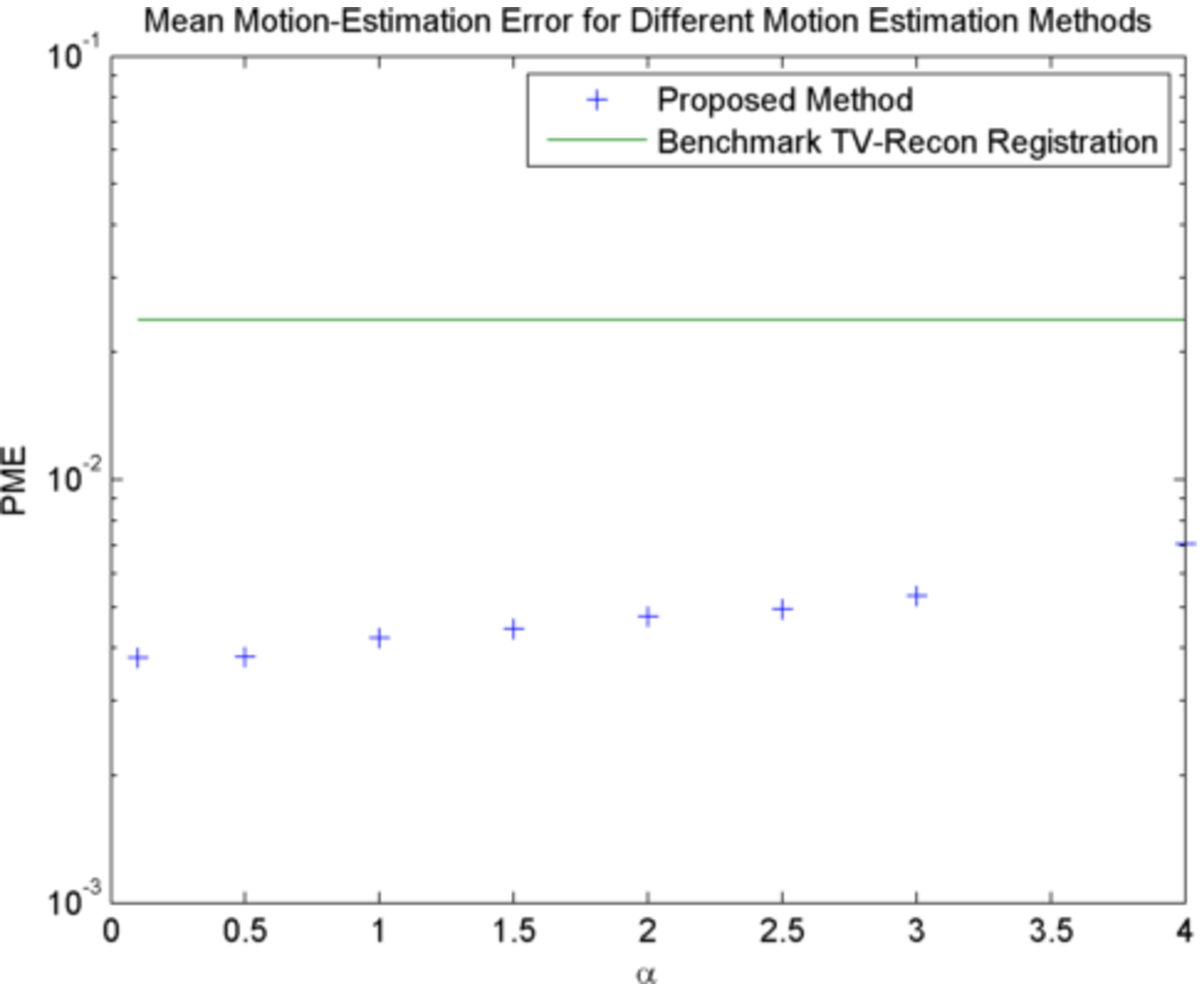}
\hspace{0.2cm}
\includegraphics[scale=0.35]{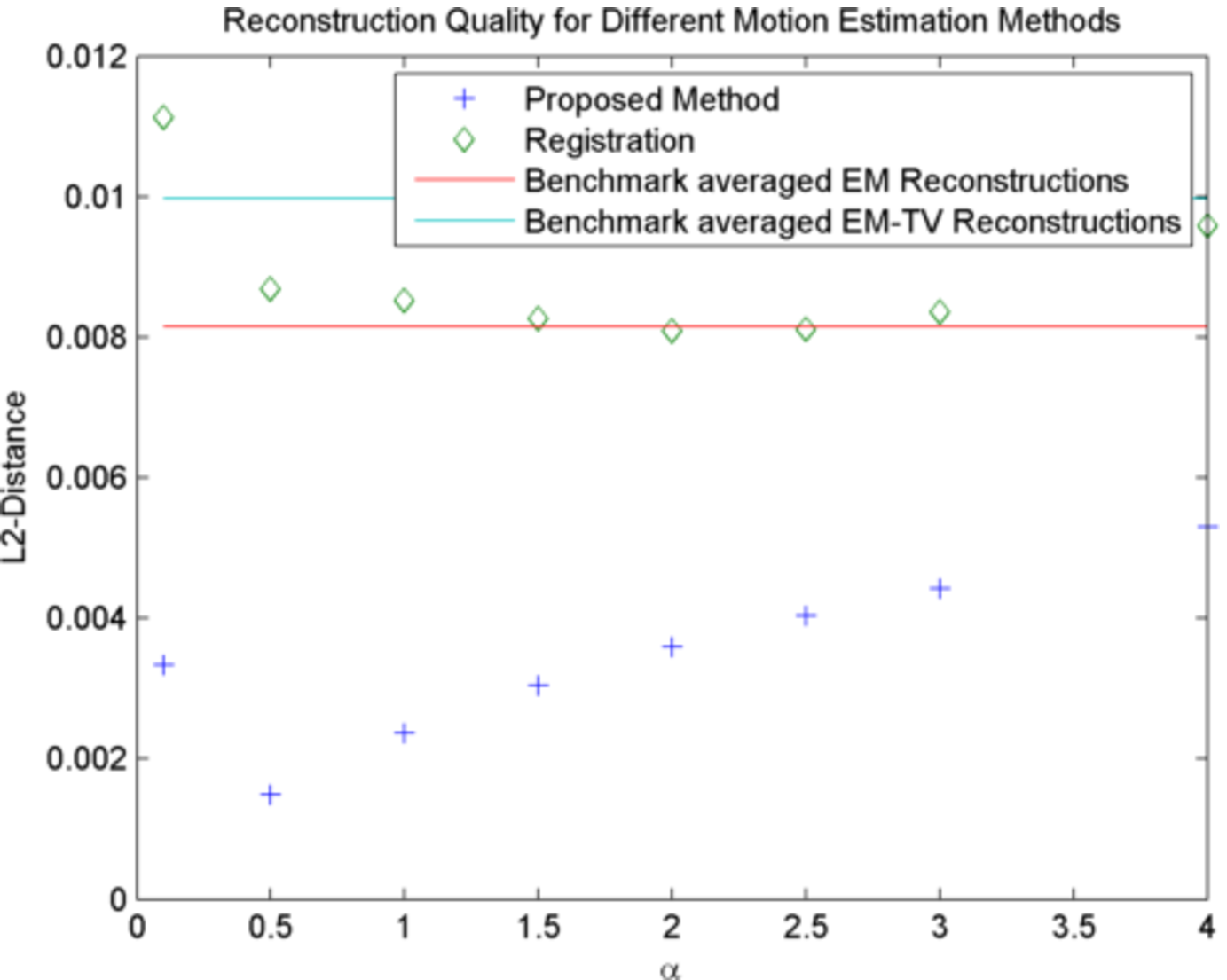}
\caption{Logarithmic plot of the averaged phantom matching errors for transformations yielded by the proposed method and registration performed on the best TV regularized single gate reconstructions as benchmark. The phantom matching errors for all transformations were averaged. The error for the proposed method is smaller by an order of magnitude. Reconstruction error for the proposed method, motion-corrected EM-TV with motion determined by registration of single gate registrations an averaged (TV regularized) single gate reconstructions as benchmark. The other motion correction techniques suffer heavily from the inaccurate motion estimation.}
\label{fig:pme_and_reg}
\end{figure}

\begin{figure}[H]
\centering
\captionsetup{justification=centering}
\subcaptionbox{Ground truth}{\includegraphics[scale=0.15]{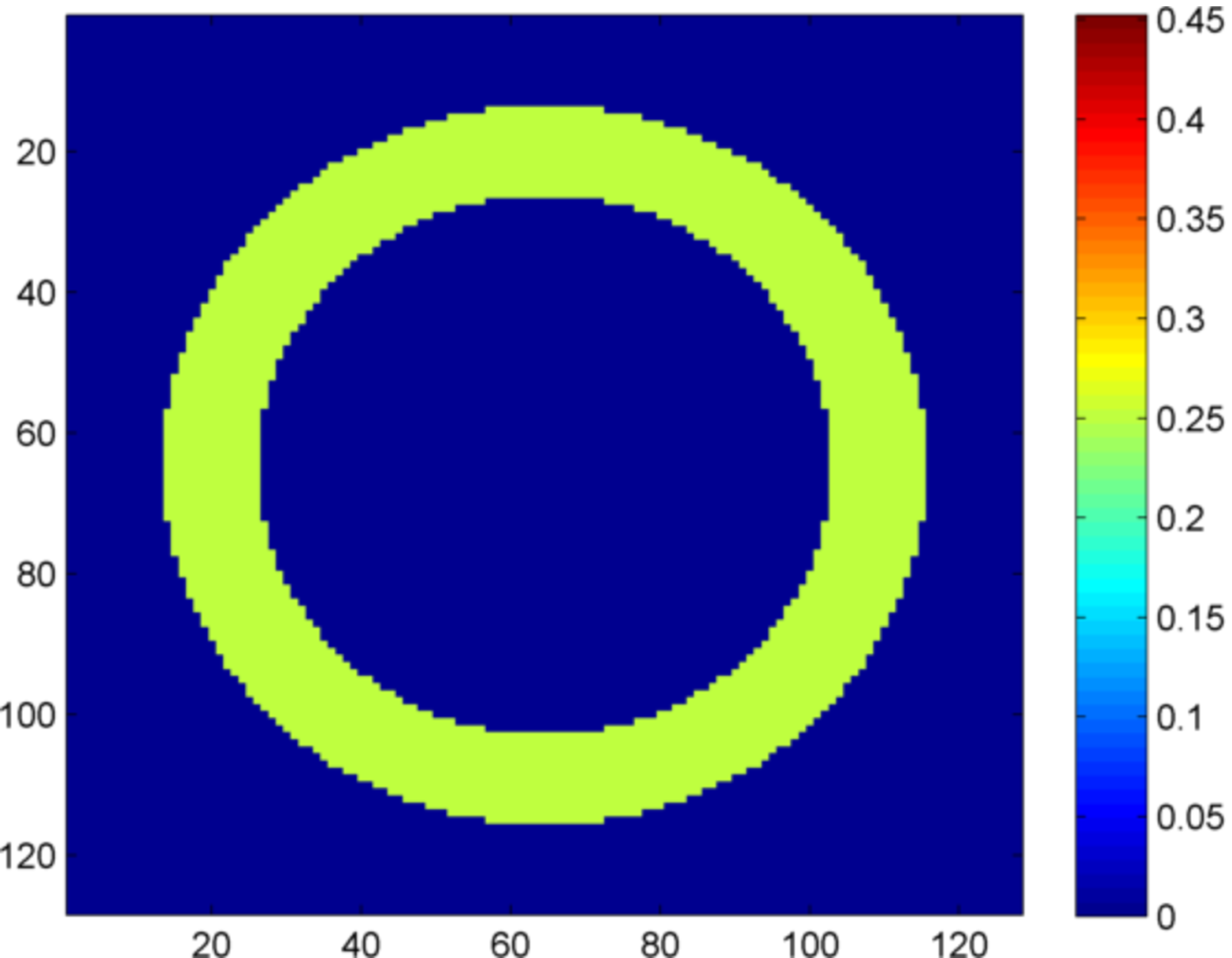}}\hspace{0.2cm}
\subcaptionbox{Averaged EM reconstruction}{\includegraphics[scale=0.15]{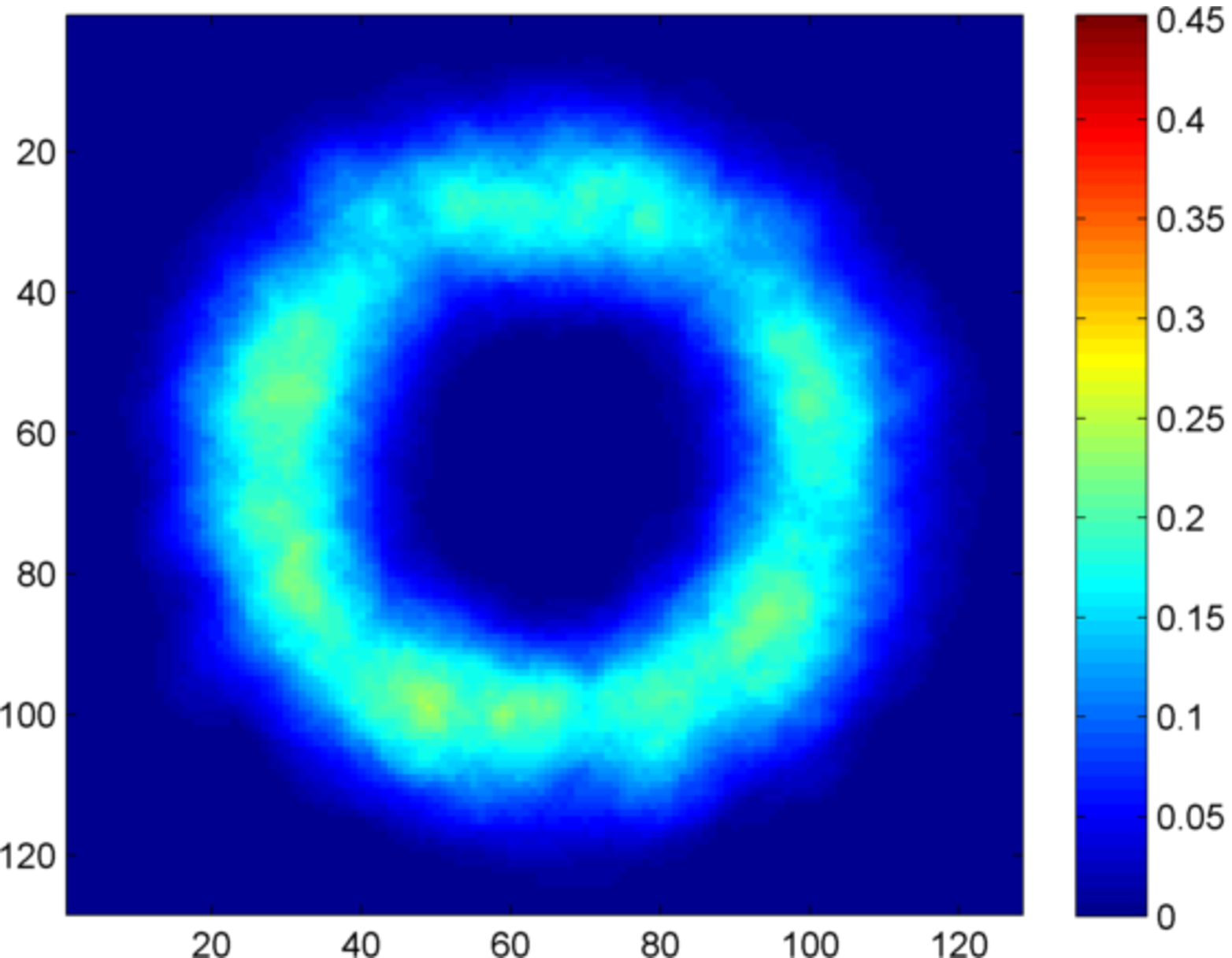}}\hspace{0.2cm}
\subcaptionbox{Averaged EM-TV reconstruction}{\includegraphics[scale=0.15]{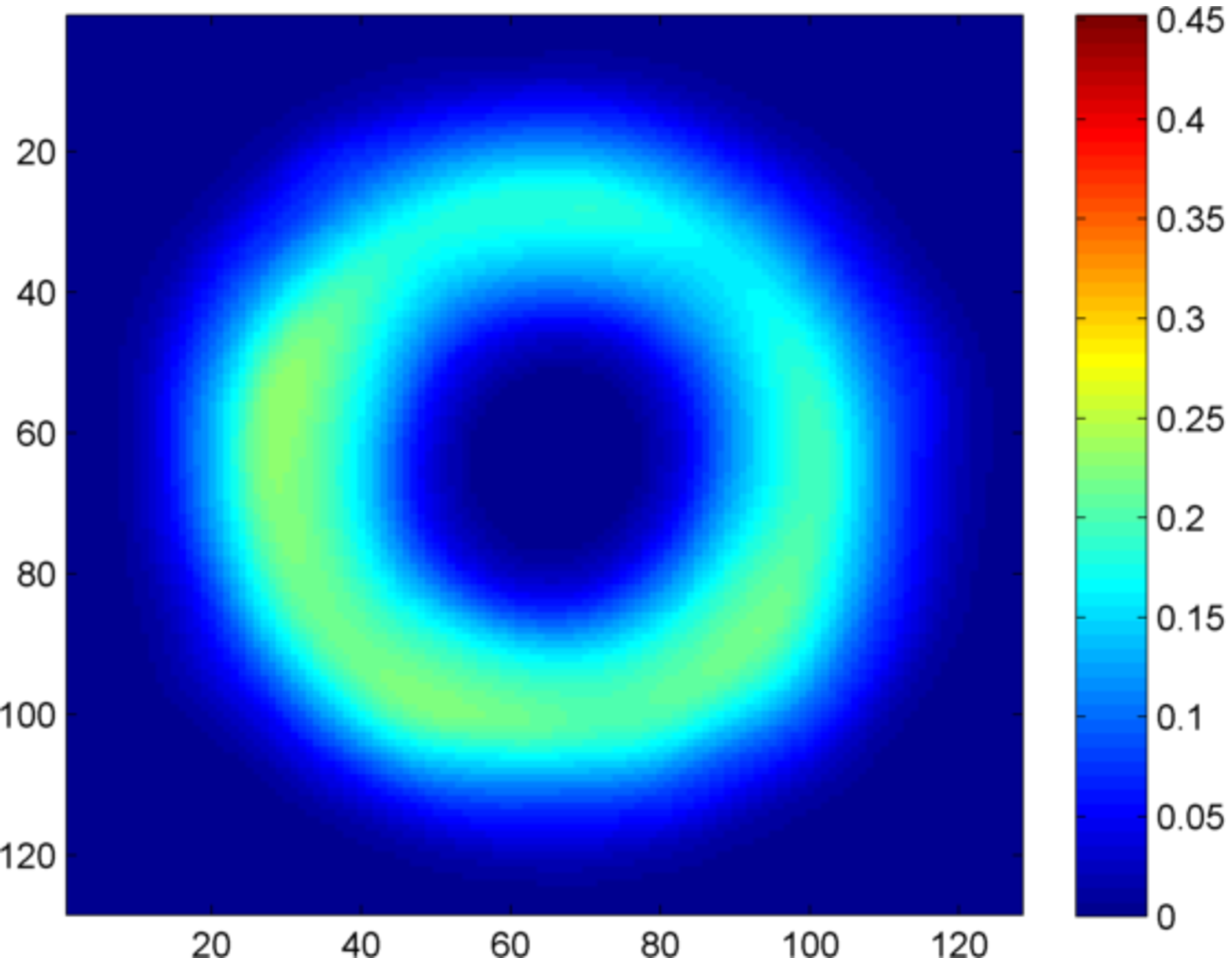}}
\hspace{0.2cm}
\subcaptionbox{Registration based \\
m.c. EM-TV}{\includegraphics[scale=0.15]{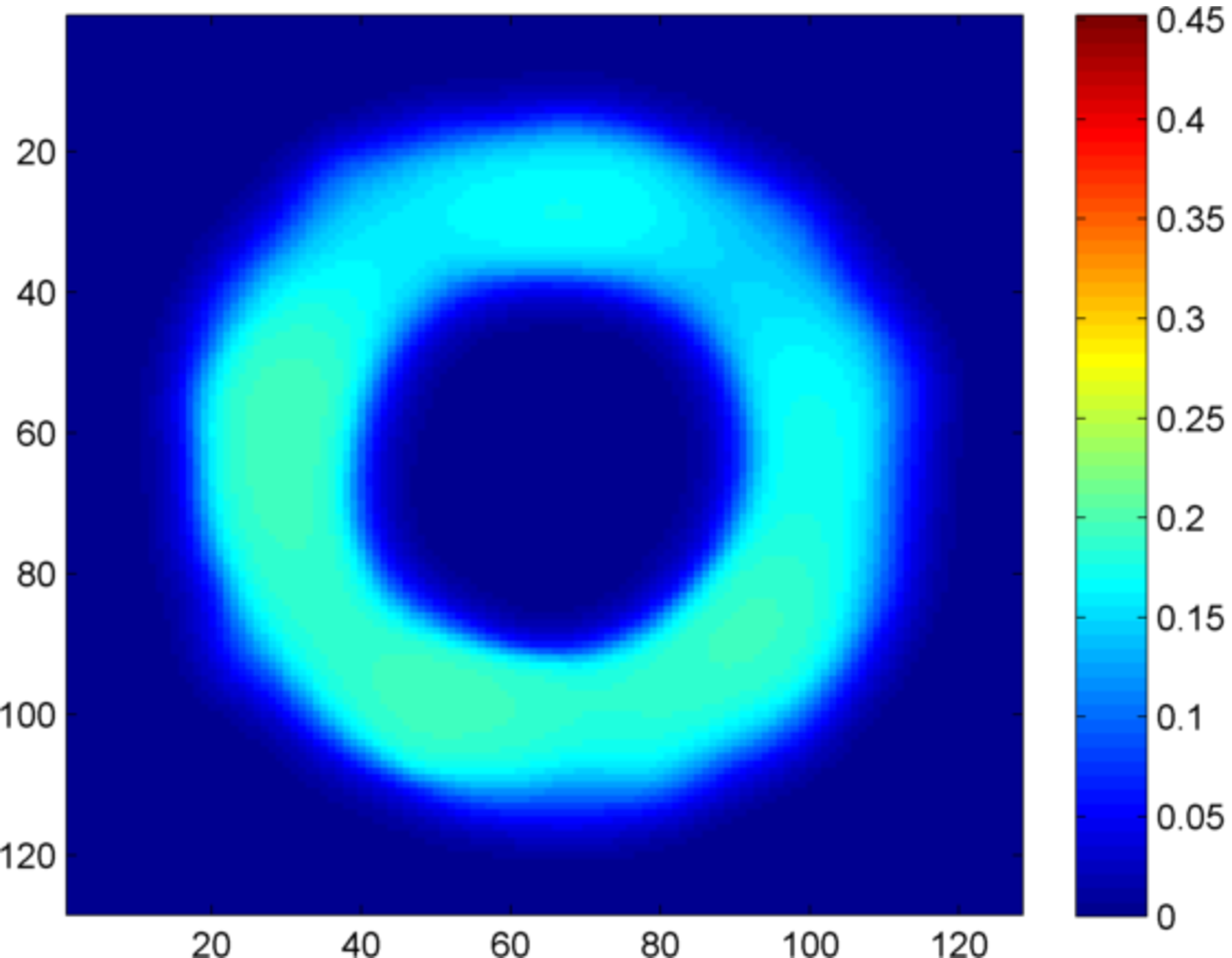}}\hspace{0.2cm}
\subcaptionbox{Proposed method}{\includegraphics[scale=0.15]{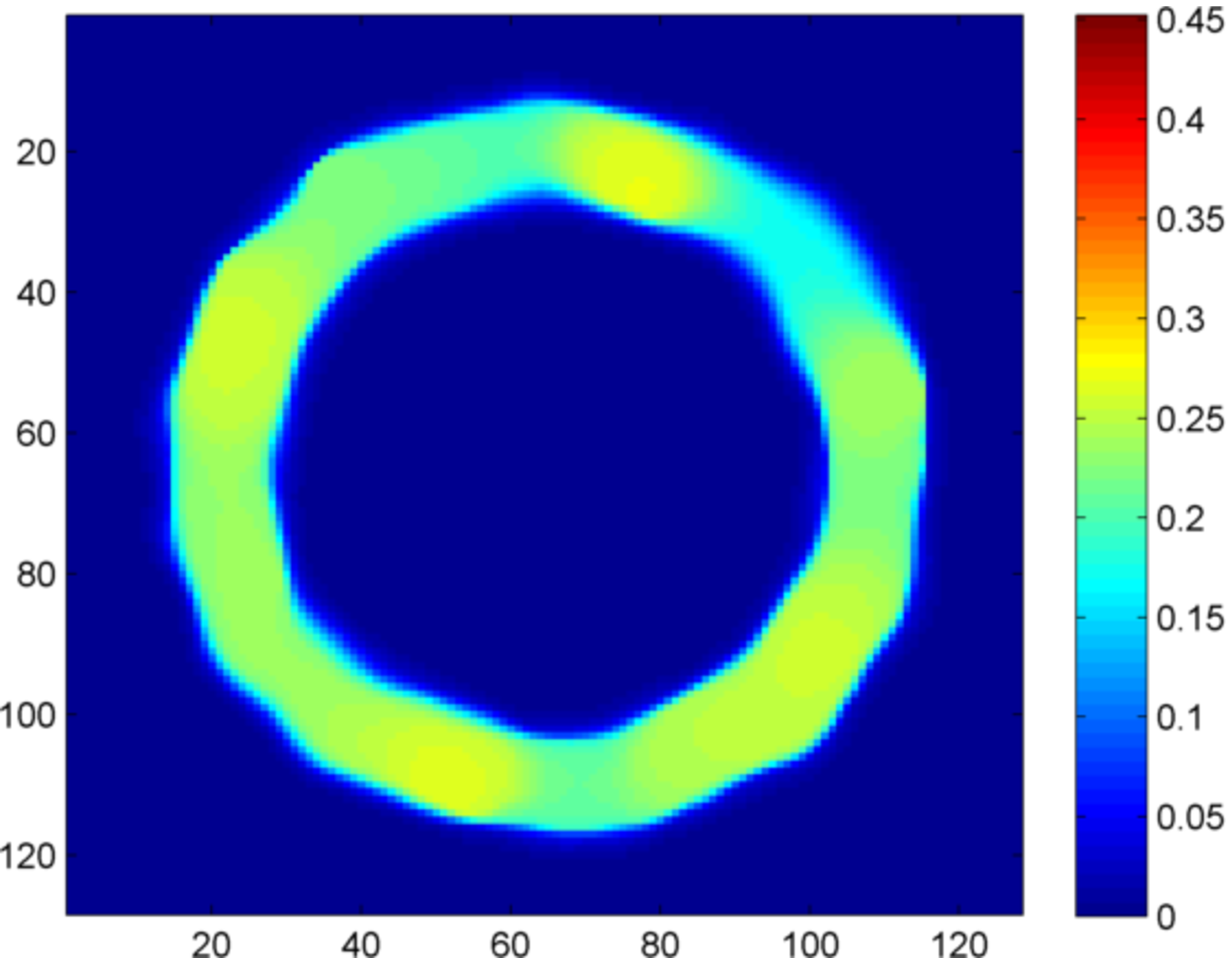}}
\captionsetup{justification=justified}
\caption{Best (TV-regularized) single gate reconstruction averaged by registration, best motion-corrected EM-TV reconstruction with motion estimation via registration of EM-TV reconstructed single gates and best reconstruction yielded by the proposed method. Note that the precise size of the ring gets lost due to the inaccurate motion estimation in the registration based methods.}
\label{fig:motion_comparison}
\end{figure}

We see that the inaccurate motion estimation by registering TV regularized single gate reconstruction has severe impact on the reconstruction quality. Since the edges are not exactly aligned by the registration, this imprecise motion estimation leads to an inexact size of the ring in the reconstruction, respectively to a blurring of the edges in the averaging process (Figure \ref{fig:motion_comparison}).

\medskip

In conclusion we can state the advantages of the proposed method are twofold. First incorporating motion-estimation directly is superior to averaging methods, since the reconstruction is performed from full data instead of just averaging images reconstructed from parts of the data. Despite small errors in the motion information the proposed methods also performs better than the single gate reconstruction with TV regularization. 
Additionally the motion-estimation yielded by the proposed method was clearly superior to the registration based methods. We can give the following three possible explanations for this:
\begin{itemize}
\item In contrast to single gate reconstructions the proposed methods transforms a template image, which is generated from the full data set.
\item By projecting the transformed template into the measurement domain the error occurring from reconstructing the reference image from part of the data is avoided.
\item The Kullback-Leibler data fidelity used in the proposed method is directly adapted to the Poisson noise characteristics of the data, while the SSD distance we use for the registration is used for Gaussian denoising.
\end{itemize} 

\subsection{XCAT Software Phantom}

In this subsection we inspect the performance of three reconstructions methods on data generated by the XCAT software phantom \cite{Segars2010}.  For generating the data we projected four cardiac gates of the phantom into the data spaces specified by the \text{Siemens Biograph Sensation 16} scanner provided by the EMRecon toolbox \cite{T.Koesters2011}. The projected data was downscaled by the factor $1000$ corrupted with Poisson noise and then scaled up again. We reconstructed the first gate from this single gate data with the classical EM algorithm, EM-TV algorithm, and the method we described in this chapter applied on the full data.

\begin{figure}[H]
\centering
\captionsetup{justification=centering}
\subcaptionbox{Ground Truth}{\includegraphics[scale=0.25]{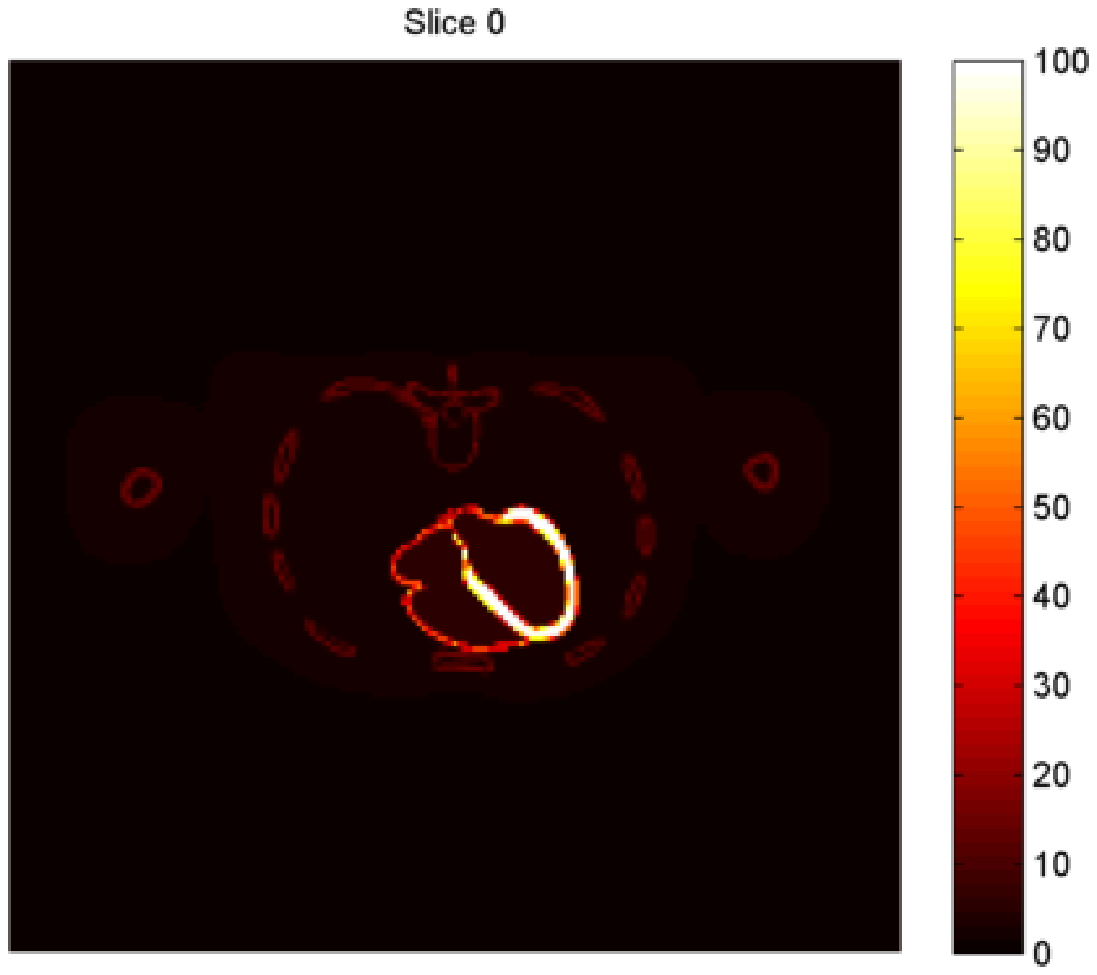}\hspace{0.5cm}}
\subcaptionbox{EM}{\includegraphics[scale=0.25]{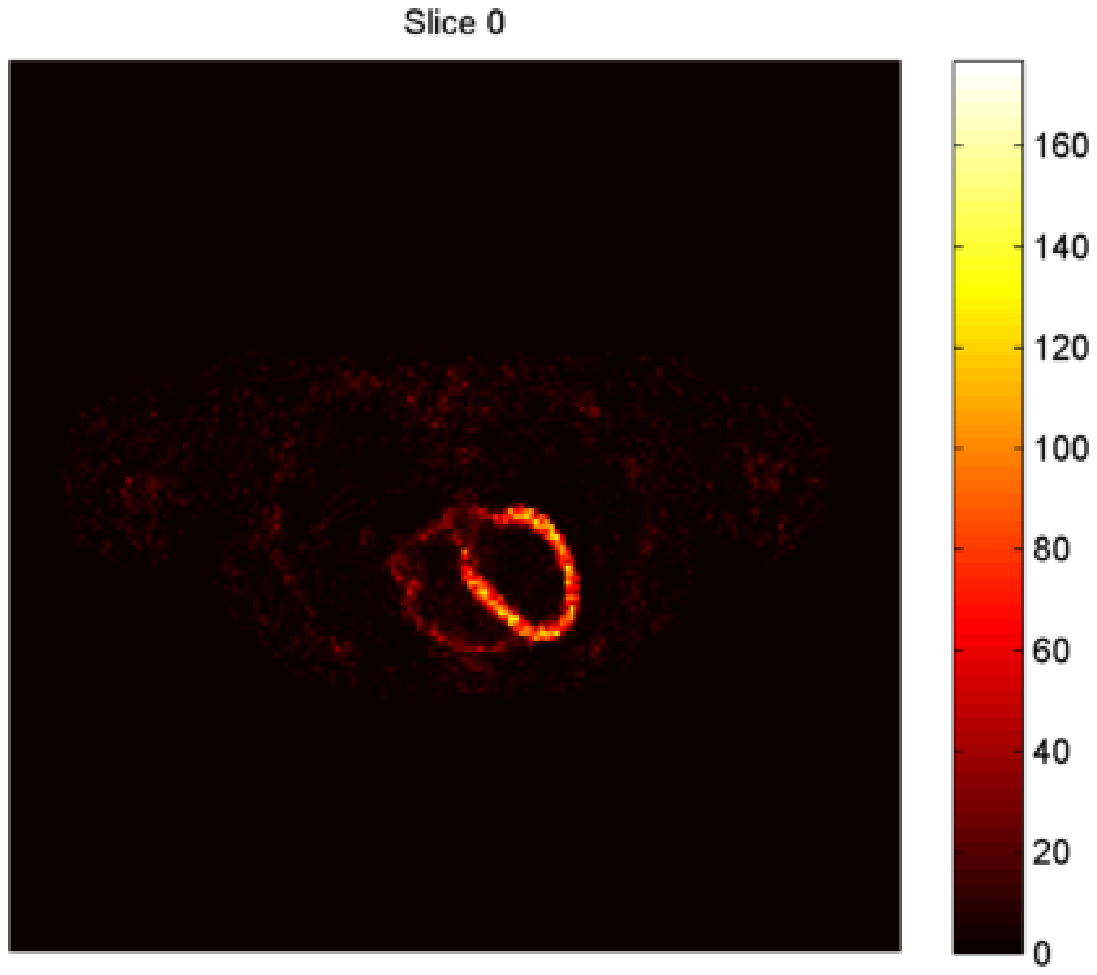}}
\hspace{0.5cm}
\subcaptionbox{EM-TV}{\includegraphics[scale=0.25]{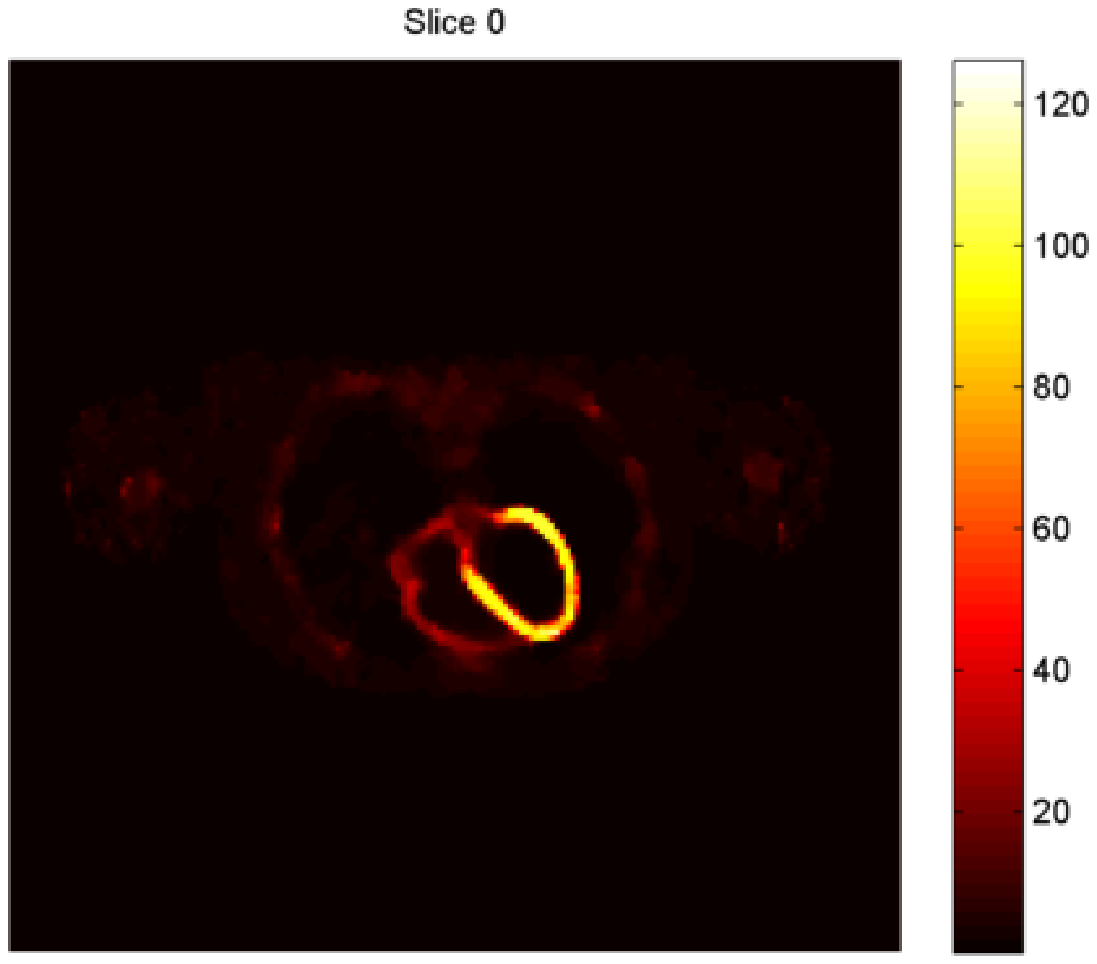}\hspace{0.5cm}}
\subcaptionbox{Proposed method}{\includegraphics[scale=0.25]{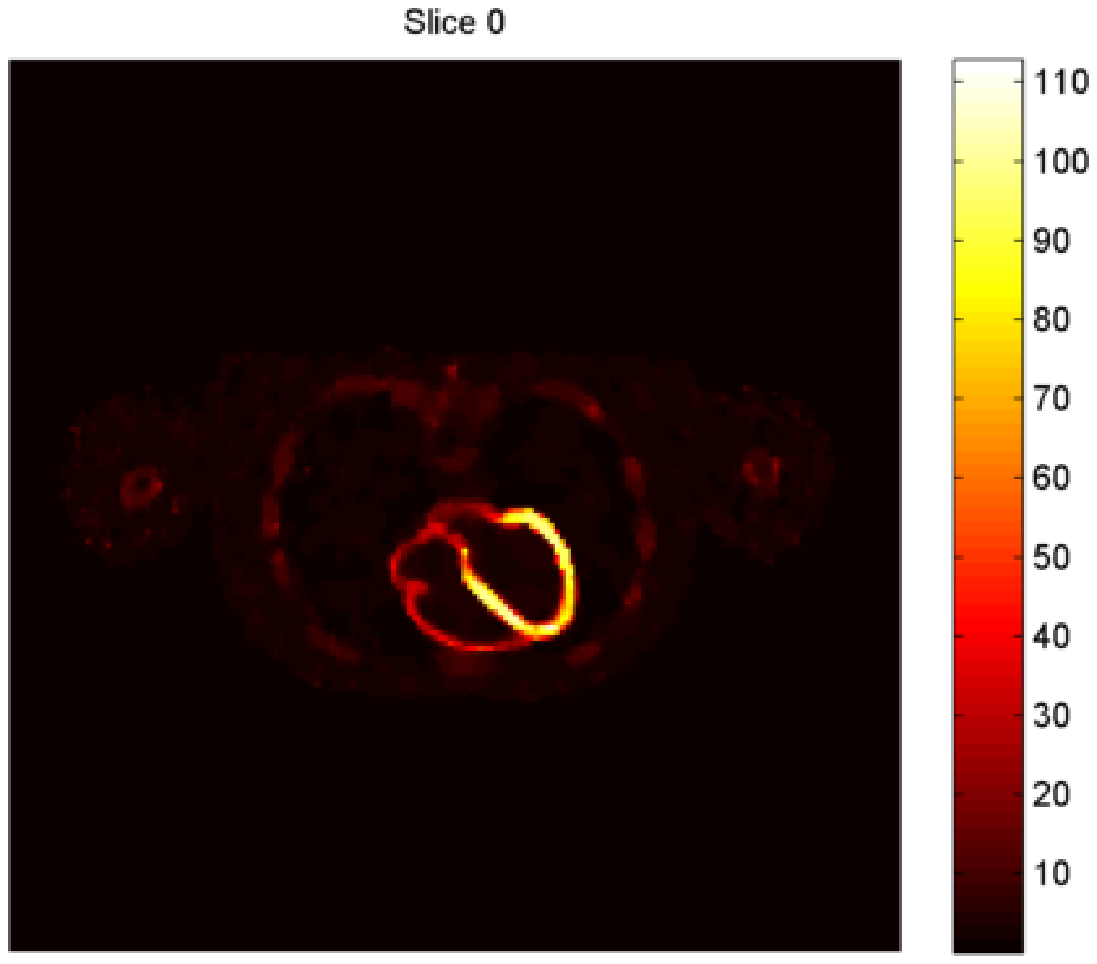}}
\captionsetup{justification=justified}
\caption{Ground truth and reconstructions EM, EM-TV ($\alpha=125$) and the proposed method ($\alpha=1000$). The TV regularized methods produce clearly better reconstruction results than the classical expectation maximization algorithm. The right ventricle is captured slightly better by the proposed method.}
\label{fig:xcat}
\end{figure}

Visual inspection of the reconstructed images shows that despite having a higher reconstruction error the proposed method captures structures with a low amount of activity better than both single gate reconstruction methods (Figure \ref{fig:xcat}). As we see the right ventricle stands out much clearer by incorporating motion information. This illustrates that the proposed method shows its potential for regions with a poor signal-to-noise ratio. 

\section{Weak Formulation of the Reconstruction Problem}

The analysis of functional \eqref{eq:functional} showed, that we needed assumptions on the injectivity resp. boundedness of the Banach indicatrix of the transformation in order to derive analytical results. We want to outline a weak formulation, which guarantees mass preservation, but does not assume injectivity of the transformation. We conclude this outlook by showing that the weak formulation implies injectivity for the motion $y$ and thus corresponds to restricting the admissible set of transformation to a certain subset of the weak diffemorphisms. 

We start by giving an equivalent formulation of the transformation model \eqref{eq:transformation_model} for a diffeomorphism $y^i$ with inverse $z^i$:

\begin{equation}
\label{eq:mcr_weak_transformation_model}
\int\limits_{\Omega^i} \rho^i(x)\varphi(x)\mathrm{d}x=\int\limits_{\Omega^i}\rho^0(y^i(x))\operatorname{det}(\nabla y^i(x))\varphi(x)\mathrm{d}x \qquad \forall \varphi \in C_0(\Omega).
\end{equation}

By applying the change of variables formula we obtain

\begin{equation}
\int\limits_{\Omega^i}\rho^i(x)\varphi(x)\mathrm{d}x=\int\limits_{\Omega}\rho^0(x)
\varphi(z^i(x))\mathrm{d}x \qquad \forall \varphi \in C_0(\Omega).
\end{equation}

Since $\rho^i \geq 0$ we can relax to Radon measures with $\mathrm{d}\mu^i$ generalizing $\rho^i\mathrm{d}x$:

\begin{equation}
\label{eq:mcr_relaxed_radon_constraint}
\int\limits_{\Omega^i}\varphi(x)\mathrm{d}\mu^i(x)=\int\limits_{\Omega}\rho^0(x) \varphi(z^i(x))\mathrm{d}x \qquad \forall \varphi \in C_0(\Omega).
\end{equation}
Alternatively we can use $\varphi(y^i(x))$ as test function:

\begin{equation}
\int\limits_{\Omega^i}\varphi(y^i(x))\mathrm{d}\mu^i(x)=\int\limits_{\Omega} \rho^0(x)\varphi(x)\mathrm{d}x \qquad \forall \varphi \in C_0(\Omega).
\end{equation}

In order to use \eqref{eq:mcr_relaxed_radon_constraint} as a constraint for minimizing \eqref{eq:functional} we provide consistency with the strong formulation by the following propsition.

\begin{proposition}
\label{proposition:mcr_weak_consistency}
Let $\rho^0$ be nonnegative and $z^i\in\mathscr{A}$ an admissible transformation. Then there exists a unique, nonnegative $\mu^i$  satisfying \eqref{eq:mcr_relaxed_radon_constraint}, which fulfils the mass-preservation property
\begin{equation}
\label{eq:mcr_weak_mass_preservation}
\int\limits_{\Omega^i}\mathrm{d}\mu^i(x)=\int\limits_{\Omega}\rho^0(x)\mathrm{d}x.
\end{equation}
\end{proposition}

\begin{proof}
Note that for any fixed $\rho^i$ and $z^i$ the right hand side of \eqref{eq:mcr_relaxed_radon_constraint} defines a linear functional. Thus there exists a unique Radon measure $\mu^i$ satisfying \eqref{eq:mcr_relaxed_radon_constraint} \cite[Chapter 1.4 Theorem 2]{Giaquinta1998a}. The mass-preservation condition \eqref{eq:mcr_weak_mass_preservation} can by shown by using test functions converging to constants.
\end{proof}

As a consequence we obtain for the weak formulation of the motion-corrected reconstruction problem:

\begin{align}
\tilde{J}(\mu,z)=&\sum\limits_{i=0}^{N}(D(K\mu^i,f^i)+\alpha^i\tv{\mu^i})+\sum\limits_{i=1}^{N}\beta^{i}S^{hyper}(z^i)\nonumber\\
\text{subject to } & \eqref{eq:mcr_relaxed_radon_constraint}
\label{eq:mcr_weak_formulation}
\end{align}

We will not elaborate further on the analysis for the weak formulation, but instead present a consistency result:

\begin{proposition}
Let $\rho^i \in L^1$ and $y^i, z^i$ fulfil \eqref{eq:mcr_weak_transformation_model} and \eqref{eq:mcr_relaxed_radon_constraint}. Then we have $N_{y^i}(\cdot,\Omega)\leq 1$ a.e. and $z^i(x)=(y^{i})^{-1}(x)$ for a.e. $x\in \Omega$.
\end{proposition}

\begin{proof}
We can deduce with the second part of the area formula (Theorem \ref{theorem:area_formula}):

\begin{align*}
\int\limits_{\Omega^i}\rho^i(x)\varphi(x)\mathrm{d}x&=\int\limits_{\Omega^i}\rho^0 (y^i(x))\operatorname{det}(\nabla y^i(x))\varphi(x)\mathrm{d}x \\
&=\int\limits_{\mathbb{R}^d}\sum\limits_{w\in((y^i)^{-1}(x)\cap \Omega)}\rho^0(y^i(w))\varphi(w)\mathrm{d}x \\
&=\int\limits_{\mathbb{R}^d}\rho^0(x)\sum\limits_{w\in((y^i)^{-1}(x)\cap \Omega)}\varphi(w)\mathrm{d}x.
\end{align*}

By using Proposition \ref{proposition:mcr_weak_consistency} and \eqref{eq:mcr_relaxed_radon_constraint} we obtain:

\begin{equation}
\int\limits_{\mathbb{R}^d}\rho^0(x)\sum\limits_{w\in((y^i)^{-1}(x)\cap \Omega)}\varphi(w)\mathrm{d}x=\int\limits_{\Omega}\rho^0(x) \varphi(z^i(x))\mathrm{d}x \qquad \forall \varphi \in C_0(\Omega).
\end{equation}

By using test functions $\varphi$ converging to constants, we can now deduce, that $N_{y^i}(\cdot,\Omega)\leq 1$. It follows that $y^i$ is weakly invertible and we can deduce
\begin{equation}
\label{eq:inverse_weak_formulation}
\int\limits_{\mathbb{R}^d}\rho^0(x) \varphi((y^i)^{-1}(x))\mathrm{d}x=\int\limits_{\Omega}\rho^0(x) \varphi(z^i(x))\mathrm{d}x \qquad \forall \varphi \in C_0(\Omega).
\end{equation}

Since \eqref{eq:inverse_weak_formulation} holds for all test functions $\varphi \in C_0(\Omega)$, the assertion follows.
\end{proof}

The proposition above guarantees that for a solution $(\mu,z)$ of the weak formulation \eqref{eq:mcr_weak_formulation}, which can be expressed via \eqref{eq:mcr_relaxed_radon_constraint} and \eqref{eq:mcr_weak_transformation_model}, the motion field $z^i$ is the weak inverse of $y^i$. Whether such an $y$ exists for a given weak solution $(\mu, z)$ is object to further research as well as a detailed analysis for the weak formulation.

\section{Discussion and Outlook}

We have presented a novel variational approach to motion corrected reconstruction of density images. After the motivation of the model with help of Bayesian statistics, we proceeded to the analysis of the model. Central part of the analysis was Theorem \ref{theorem:konvergenz}, which ensured us weak $L^{1}$-convergence of the sequence of transformated images, where images as well as transformations were sequences. To prove this theorem we relied heavily on regularity properties granted by the regularizers for intensity and motion vectors. A critical point was the boundedness of the Banach indicatrix in $L^{\infty}$. It remains unclear, if we can establish a bound of the form
\begin{equation}
\|N_y(\cdot,\Omega)\|_{\infty} \leq C(\Omega)\mathscr{S}^{hyper}(y).
\end{equation}
In order to give existence results independent of this assumption, we restricted ourselves to the case of injective transformations. Although this restriction can be motivated by a mass-preservation demand, which is in our model only granted for injective transformations, the numerical realization can again be challenging. We proposed a framework with injective Dirichlet-boundary conditions to guarantee injectivity for the whole domain, but an extension to a less restrictive model might be of interest (see e.g. \cite{Ashburner2007} for a diffeomorphic registration framework).

We thoroughly tested the method on an artificial deblurring example and showed superiority to several other reconstruction methods. By applying the method on  software-phantom data we gave a proof-of-concept for the applicability for real data. Although there are still difficulties to deal with the movement of really small objects the method seems to be well-suited for motion-corrected reconstruction of clinical data, especially with a low signal-to-noise ratio. Again a thorough evaluation on clinical data is the next step to go, as well as finding new means to assess the quality of a found transformation in motion correction.

Motion-corrected reconstruction with alternating minimization has been done e.g. by Mair et al. \cite{B.A.Mair2006}, although neither the concept of mass-preservation nor TV regula\-rization is imposed. To the best of our knowledge the work of Blume et al. \cite{Blume2010TMI} is most closely related to the presented framework: In \cite{Blume2010TMI} the authors propose a similar general framework to reconstruction with simultaneous motion estimation and incorporate a local invertibility constraint in \cite{Blume2012}. While the focus of Blume et al. lies on the actual implementation of a reconstruction method with a parametric B-spline transformation model, our main contribution are:

\begin{itemize}
\item A thorough analysis of the motion-corrected reconstruction problem with mass-preserving transformation model and appropriate regularization for density image and motion.
\item A framework for the numerical solution of the motion corrected reconstruction problem.
\end{itemize}

In the following we present some open questions, which can motivate further research in this field.

\begin{itemize}
\item Incorporation of attenuation correction: In general there is only an attenuation map for one gate available, so it is a possibility to deform said map for the reconstruction of the other gates. Challenges lie in the implementation as well as the analysis.
\item Incorporating of a priori information into the reconstruction framework. This can be either structural information via MR images \cite{Ehrhardt2015} or information on the motion as boundary values \cite{DirkMannweiler2014} or estimated motion from other modalities \cite{Qiao2006}.
\item Another open problem are the convergence properties of the proposed alternating minimization algorithm: Due to Beck \cite{Beck2015} we can guarantee that the sequence generated by our algorithm has a stationary point as accumulation point, if we impose some reasonable conditions on discretization and interpolation. Although we cannot expect convergence to a global minimum due to the nonconvexity of the problem, proximal regularized minimization algorithms \cite{Attouch2010} might improve the performance.
\end{itemize}

\appendix

\section{Weak Diffeomorphisms}

We will present some basic ideas and theorems from \cite{Giaquinta1994}, how to generalize diffeomorphisms for Sobolev mappings, which are not necessarily differentiable. Giaquinta's central idea for a mapping
\begin{equation}
y : \Omega\subset\mathbb{R}^{d} \rightarrow \mathbb{R}^m
\end{equation}

is to use properties of the graph $\mathcal{G} \subset \mathbb{R}^{d}\times \mathbb{R}^{m}$. Based on this idea he introduced Cartesian currents \cite{Giaquinta1989}. As the description of these currents is beyond the scope of this paper, we will only present the results from \cite{Giaquinta1994} related to weak diffeomorphisms and refer to \cite{Giaquinta1994,Giaquinta1998a,Giaquinta2010} for a detailed course on Cartesian currents. Because we do not want to discuss the theory on Cartesian currents, we define a sufficient class of transformations, which contains transformations fulfilling some rather complicated requirements from the theory of Cartesian currents.

\begin{definition}
We define the following two subclasses of Sobolev mappings:
\begin{align*}
A_{p,q}(\Omega):=&\{y\in W^{1,p}(\Omega;\mathbb{R}^m)\,|\,\operatorname{cof}(\nabla y)\in L^q\},\\
A_{p,q}^{+}(\Omega):=&\{y\in A_{p,q}(\Omega)\,|\,\operatorname{det}(\nabla y) > 0 \quad\text{ a.e.}\}.
\end{align*}
\end{definition}

Next we introduce weak inverses as in \cite{Giaquinta1994}:
\begin{definition}[Weak inverse]
\label{defi:weak_inverse}
Given a measurable map
\begin{equation} 
y : \Omega  \rightarrow \hat{\Omega} \qquad \lambda(\Omega)>0 \quad\lambda(\hat{\Omega})>0.
\end{equation}
We say that
\begin{itemize}
\item[1. ] $y$ is weakly invertible with weak inverse $\hat{y}$, if and only if
\end{itemize}
\begin{alignat}{2}
&\hat{y}(y(x))=x \qquad & \text{for almost every } x\in \Omega , \\
&y(\hat{y}(z))=z \qquad & \text{for almost every } z\in \hat{\Omega}.
\end{alignat}

\begin{itemize}
\item[2. ] $y$ is a weak one-to-one transformation, iff there exists a measurable map
\end{itemize}
\begin{equation*}
\hat{y}:\hat{\Omega} \rightarrow \Omega,
\end{equation*}
\begin{itemize}
\item[] such that
\end{itemize}

\begin{center}
\begin{itemize}
 \item[a)] $y$ and $\hat{y}$ fulfill Lusin's condition \eqref{eq:lusin_condition}.
\item[b)] $y$ and $\hat{y}$ are the inverses of the respective other.\\
\end{itemize}
\end{center}
\end{definition}

The next theorem provides some properties of the inverse of a mapping:

\begin{theorem}
Let $y:\Omega \rightarrow \hat{\Omega}$ be a weakly invertible map with inverse $\hat{y}$. Suppose that
\begin{itemize}
\item[(i)] $y$ satisfies Lusin's condition \eqref{eq:lusin_condition},
\item[(ii)] $y$ is almost everywhere approximately differentiable in $\Omega$.
\end{itemize}
Then $\hat{y}$ is approximately differentiable almost everywhere in $\hat{\Omega}$. Moreover:
\begin{alignat}{2}
&Dy(\hat{y}(z))D\hat{y}(z)=Id_{\hat{\Omega}} \qquad & \text{for } a.e.\quad z \in \hat{\Omega} \\
&D\hat{y}(y(x))Dy(x)=Id_{\Omega} \qquad & \text{for } a.e. \quad x\in \Omega  
\end{alignat}
\end{theorem}
\begin{proof}
See \cite[Chapter 3, Theorem 2]{Giaquinta1994}.
\end{proof}

With this at hand Giaquinta et al. \cite{Giaquinta1994} define global invertibility by using properties of the graph of a map.

\begin{definition}[Global invertibility for a.e. approximately differentiable mappings]
Let $y$ be an a.e. approximately differentiable map from $\Omega\subset \mathbb{R}^{d}$ into $\mathbb{R}^m$ with $\operatorname{det}(Dy)\in L^1(\Omega)$ satisfying 
\begin{equation}
\operatorname{det}(Dy) > 0 \qquad \text{a.e. in } \Omega .
\end{equation}
We say $y$ is globally invertible if and only if 
\begin{equation}
\label{eq:global_invertibility_defi}
\int\limits_{\Omega}\phi(x,y(x))\operatorname{det}(Dy(x))\mathrm{d}x\leq\int\limits_{\mathbb{R}^m}\left(\sup\limits_{x\in\Omega}\phi(x,z)\right)\mathrm{d}z
\end{equation}
holds for all $\phi \in C_{c}^{0}(\Omega\times \mathbb{R}^m)$ with $\phi \geq 0$.

Consequently we denote the set of injective functions by $\mathscr{I}(\Omega)$.
\end{definition}

The definition above is related to the area formula; for any function $\phi \in C_{c}^{0}(\Omega\times \mathbb{R}^m)$ we observe:

\begin{equation*}
\int\limits_{\Omega}\phi(x,y(x))\operatorname{det}(Dy)(x)\mathrm{d}x\leq\int\limits_{\Omega}\underbrace{\sup\limits_{x\in\Omega}\phi(x,y(x))}_{=:\psi(y(x))}\operatorname{det}(Dy(x))\mathrm{d}x=\int\limits_{\mathbb{R}^m}\sup\limits_{x\in\Omega}\phi(x,z)N_y(z,\Omega)\mathrm{d}z
\end{equation*}

This illustrates, how the invertibility condition \eqref{eq:global_invertibility_defi} can be violated by functions which are not injective on sets with positive measure.  According to \cite{Giaquinta1994} global invertibility can be defined equivalently in several other ways:

\begin{proposition}
Let $y:\Omega \subset\mathbb{R}^d \rightarrow \mathbb{R}^m$ be a.e. approximately differentiable in $\Omega$ with $\operatorname{det}(Dy)\in L^1(\Omega)$  and $\operatorname{det}(Dy)\geq 0$ a.e. in $\Omega$. Then the following claims are equivalent:

\begin{itemize}
\item[(i)] $y$ is globally invertible.
\item[(ii)] For any $\phi \in C_c^0(\mathbb{R}^m)$ with $\phi\geq 0$ y satisfies the inequality 
\end{itemize}
\begin{equation}
\int\limits_{\Omega}\phi(y(x))\operatorname{det}(Dy)(x)\mathrm{d}x\leq \int\limits_{\mathbb{R}^m}\phi(z)\mathrm{d}z. 
\end{equation}

\begin{itemize}
\item[(iii)] For almost every $z\in\mathbb{R}^m$ we have
\end{itemize}
\begin{equation}
N(y,\Omega,z)\leq 1.
\end{equation}
\begin{itemize}
\item[(iv)] For almost every $z\in \mathbb{R}^m$ we have
\end{itemize}
\begin{equation}
N_y(\Omega,z)=\chi_{y(\Omega)}(z):=\begin{cases}1 & z \in y(\Omega)\\ 0 &z \notin y(\Omega) \end{cases}.
\end{equation}
\begin{itemize}
\item[(v)] We have
\end{itemize}
\begin{equation}
\int\limits_{\Omega} \operatorname{det}(D y)(x)\mathrm{d}x=\mathscr{H}^m(\tilde{y}(\Omega))
\end{equation}
\begin{itemize}
\item[] where $\mathscr{H}^m$ is the m-dimensional Hausdorff measure and $\tilde{y}$ a Lusin representative.
\item[(vi)]The inequality
\end{itemize}
\begin{equation}\int\limits_{\Omega} \operatorname{det}(D y)(x)\mathrm{d}x\leq\mathscr{H}^m(y(\Omega))
\end{equation}
\begin{itemize}
\item[] holds for any representative of $y$.
\end{itemize}
\end{proposition}
\begin{proof}
See \cite[Chapter 5, Proposition 1]{Giaquinta1994}
\end{proof}

Having this in mind we can now define a norm for the class of weak diffeomorphisms:

\begin{definition}[Norm for weak diffeomorphisms]
For any almost everywhere approximately differentiable map $y:\Omega\subset \mathbb{R}^d \rightarrow \mathbb{R}^m$ we set
\begin{equation}
|M(Dy)|:=\left(1+|Dy|^2+|\operatorname{cof}(Dy)|^2+|\operatorname{det}(Dy)|^2\right)^{\frac{1}{2}}
\end{equation}
and define
\begin{equation}
\|y\|_{\operatorname{dif}^{p,q}}:=\int\limits_{\Omega}\left(|y|^p+|M(Dy)|^p+\frac{|M(Dy)|^q}{|\operatorname{det}(Dy)|^{q-1}}\right)\mathrm{d}x.
\end{equation}

\end{definition}

Now we can define the class of weak diffeomorphisms:

\begin{definition}[Space of weak diffeomorphisms]
We say that a map $y:\Omega\subset \mathbb{R}^d \rightarrow \mathbb{R}^m$ belongs to the class $\tilde{\operatorname{dif}}^{p,q}(\Omega,\mathbb{R}^m)$ for $p,q\geq 1$, if and only if:
\begin{enumerate}
\item  $|M(Dy)|\in L^p$.
\item $y$ has a closed graph in $\Omega \times \mathbb{R}^m$.
\item $\operatorname{det}(Dy)>0$ a.e. in $\Omega$.
\item $y$ is globally invertible.
\item $\|y\|_{\operatorname{dif}^{p,q}}<\infty$.
\end{enumerate}

A map only fulfilling the first three conditions is called a  weak local diffeomorphism.
\end{definition}

\begin{rem}
Since giving a thorough definition of the second property requires insight in the theory of Cartesian currents, we will not elaborate further on this subject. Details can be found e.g. in \cite{Giaquinta1989,Giaquinta1998a,Giaquinta2010}. However we can state that $y\in A_{d-1,\frac{d}{d-1}}$ is sufficient \cite{Giaquinta1994a} but not necessary \cite{Giaquinta1998} to guarantee the closedness of the graph.
\end{rem}

\end{document}